\documentclass[twoside]{article}
\usepackage{amsmath, amsthm, amssymb, amsfonts, mathrsfs, mathtools}
\usepackage{graphicx}
\usepackage{makeidx}
\usepackage[left=2cm,right=2cm,top=2cm,bottom=2cm]{geometry}
\usepackage{caption}
\usepackage{subcaption}
\usepackage{tikz}
\usetikzlibrary{decorations.pathreplacing}
\tikzstyle{vertex}=[auto=left,circle,draw=black,fill=white, inner sep=1.5]
\usepackage{fancyhdr,lipsum}

\newtheorem{theorem}{Theorem}[section]

\newtheorem{conj}[theorem]{Conjecture}
\newtheorem{lemma}[theorem]{Lemma}
\newtheorem{prop}[theorem]{Proposition}

\newtheorem{definition}{Definition}

\title{Chromatic Index of Signed Generalized Book Graphs and Signed Complete Graphs}

\author{$^{\ast}$Deepak Sehrawat\\
Department of Mathematics\\
Pandit Neki Ram Sharma Government College Rohtak\\
Rohtak - 124001, India\\
Email: deepakssehrawat5@gmail.com\\
\\  Rohit\\
Department of Mathematics, Maharshi Dayanand University Rohtak\\
Rohtak - 124001, India\\
Email: rohit.rs24.maths@mdurohtak.ac.in\\
}

\date{}
\begin{document}
\maketitle
\thispagestyle{plain}
\vspace{-0.3in}

\begin{center}{Abstract}\end{center}
A signed graph $(G,\sigma)$ consists of a graph $G$ and the signature $\sigma : E(G) \rightarrow \{+1,-1\}$. An incidence of $G$ is a pair $(v,e)$, where $v$ is one of the end vertices of an edge $e \in E(G)$. A proper $q$-edge coloring $\gamma$ of signed graph $(G,\sigma)$ is an assignment of colors to incidences satisfying that $\gamma(v,e) = - \sigma(e) \gamma(w,e)$ for every edge $e=vw$ and for any two incidences $(v,e)$ and $(v,f)$, involving the same vertex, $\gamma(v,e) \neq  \gamma(v,f)$. The chromatic index of a signed graph $(G,\sigma)$, denoted by $\chi'(G,\sigma)$, is the minimum number $q$ for which $(G,\sigma)$ has a proper $q$-edge coloring. In this paper, we determine the chromatic index of signed generalized book graphs. We also determine the chromatic index of signed complete graphs of order up to six. 

\vspace{0.1in}
\noindent
{\textbf{Keywords:} Chromatic index, complete graph, edge coloring, generalized book graph, signed graph.

\noindent {\textbf{2020 Mathematics Subject Classification:} 05C15, 05C22.
 
\section{Introduction}\label{intro}
The main goal of this paper is to determine the chromatic index of all signed generalized book graphs and all signed complete graphs of order up to six.

All graphs considered in this paper are finite, simple and undirected graphs. For a graph $G$, $V(G)$ and $E(G)$ denote the {\it vertex set} and {\it edge set} of $G$, respectively. The {\it degree} of a vertex $v \in V(G)$, denoted $\deg_{G}(v)$, is the number of edges incident to $v$. Given a graph $G$, the {\it maximum degree} $\Delta(G)$ of $G$ is the maximum degree among the degrees of all vertices in $G$.

A proper $q${\it -edge coloring} of a graph $G$ is a mapping $\gamma : E(G) \rightarrow \{1, \ldots , q\}$ such that adjacent edges receive distinct colors. The {\it chromatic index} of a graph $G$, denoted by $\chi'(G)$, is the minimum number $q$ required for a proper $q$-edge coloring of $G$. In \cite{Vizing1964}, Vizing proved that, for every graph $G$, either $\chi'(G) = \Delta(G)$ or $\chi'(G)= \Delta(G) +1$. Moreover, a graph $G$ is {\it class 1} if $\chi'(G) = \Delta(G)$ and {\it class 2} if $\chi'(G)= \Delta(G) +1$.

Let $G$ be a graph and a mapping $\sigma : E(G) \rightarrow \{+1,-1\}$. Then a pair $\Sigma=(G,\sigma)$ is called a {\it signed graph}. Given a signed graph $\Sigma$, $G$ is called the {\it underlying graph} and $\sigma$ is called the {\it signature} of $\Sigma$.

Vertex coloring of signed graphs was initiated by Zaslavsky~\cite{Zaslavsky1982}. His idea of coloring the vertices of a signed graph had compatibility with deletion/contraction recurrence and chromatic polynomials that specializes to ordinary graphs when the signed graph is all positive. In 2016, Máčajová et al. \cite{Máčajová2016} used the idea of color set, given by Zaslavsky, to define the chromatic number of a signed graph and proved an extension of famous Brook's theorem in the context of signed graphs. The concept of edge coloring of signed graphs was independently proposed by Behr~\cite{Behr2020} and Zhang et al. \cite{Zhang2020}. In this paper, our computation is carried out along the definition of edge coloring introduced by Behr. 

In \cite{Cai2022}, authors computed the chromatic index of signed generalized Petersen graph $P(n,1)$ for $n \geq 5$. Particularly, they proved that $\chi'(P(n,1),\sigma)=3$ for $n \geq 5$ and gave several examples satisfying $\chi'(P(5,2),\sigma)=4$ and $\chi'(P(6,2),\sigma)=4$. Zheng et al. \cite{Zheng2022} considered the chromatic index of signed generalized Petersen graph $P(n,2)$ and proved that $\chi'(P(n,2),\sigma)=3$ if $n \equiv 3\mod 6 (n \geq 9)$ and $\chi'(P(n,2),\sigma)=4$ if $n=2p (p \geq 4)$. Recently, Wen et al. studied the edge coloring of Cartesian product of signed graphs in \cite{Wen2025}.

Let $m$ and $n$ be positive integers such that $m\geq 3$ and $n \geq 2$. The $m$-cycle {\it book graph} $B(m,n,2)$ consists of $n$ copies of the cycle $C_m$ whose intersection is a path $P_2$. For $m=3$, the book graph $B(3,n,2)$ is well-known {\it triangular book} graph. Shi and Song \cite{Shi2007} obtained upper bounds on the spectral radius of triangular book-free graphs. Sehrawat and Bhattacharjya \cite{Sehrawat2022} computed the chromatic number and chromatic polynomials of all signed book graphs. However the chromatic index of signed book graphs is still unknown. We generalize the family of book graph as follows.

\begin{definition}
\rm{For integers $m \geq 3$, $n \geq 2$ and $k \geq 2$, the {\it generalized book graph} $B(m,n,k)$ consists of $n$ copies of the cycle $C_m$ whose intersection is a path $P_k$.}
\end{definition}

In Subsection 3.1, we will compute the chromatic index of all signed generalized book graphs. 
 
To the best of our knowledge, the edge coloring of signed complete graphs is also not studied anywhere. However, for unsigned complete graphs, it is known that the chromatic index of a complete graph $K_n$ depends on the parity of $n$. More precisely, 

$$ \chi'(K_n) = 
 \begin{cases}
 n-1,~~~~~\text{if}~ n~\text{ is even};\\
n,~~~~~~~~~~~ \text{if}~ n~\text{ is odd}.
\end{cases}
$$ 

In Subsection 3.2, we consider the edge coloring of signed complete graphs of order up to six. We determine the chromatic index of all signed complete graphs of order up to six. We also show that for any signature $\sigma$, $\chi'(K_n,\sigma) =n-1 = \chi'(K_n)$, when $2 \leq n \leq 6$ is even. It is also proved that for odd values of $n$, unlike unsigned case, there are some signed complete graphs $(K_n,\sigma)$ for which $\chi'(K_n,\sigma) =n-1 < \chi'(K_n)$. The paper ends with some conjectures and open problems.

\section{Preliminaries}\label{prelim}
In this section, we present some necessary definitions, notations and results which are used to obtain our main results. 

In a signed graph $(G,\sigma)$, an edge $e \in E(G)$ is {\it positive} (respectively, {\it negative}) if $\sigma (e) = 1$ (respectively, $\sigma (e) = -1$). Let $\sigma^{-1}(-1) = \{e \in E(G)~:~\sigma(e)=-1\}$, then $(G,\sigma)$ is {\it all-positive} if $\sigma^{-1}(-1) = \emptyset$ and {\it all-negative} if $\sigma^{-1}(-1) = E(G)$. By $(G,+)$ and $(G,-)$, we denote all-positive and all-negative signed graphs, respectively. A cycle is {\it positive} if the product of its edge signs is positive and {\it negative}, otherwise. A signed graph is {\it balanced} if its all cycles are positive, and {\it unbalanced}, otherwise. The terms signed graph and its balance appeared first in a paper of Frank Harary~\cite{Harary1955}. 

Let $v$ be a vertex of a signed graph $(G,\sigma)$. Then {\it Switching} $v$ in $(G,\sigma)$ is an operation that changes sign of each edge incident to $v$. In general, if we switch $X \subseteq V(G)$ in $(G,\sigma)$, then we get a signed graph $(G,\sigma')$ such that for every edge $uv \in E(G)$ we have $$ \sigma'(uv) = 
 \begin{cases}
 -\sigma(uv),~~~~~~\text{if exactly one of}~ u, v ~\text{belongs to}~ X, \\
\sigma(uv),~~~~~~~~ \text{otherwsise}.
\end{cases}
$$ 

Two signed graphs $(G,\sigma')$ and $(G,\sigma)$ are {\it switching equivalent} (or, simply {\it equivalent}) if $(G,\sigma')$ can be obtained by switching some of the vertices of $(G,\sigma)$. It is denoted by $(G,\sigma') \sim (G,\sigma)$ (or $\sigma' \sim \sigma$ if $G$ is clear from the context). The following characterization for two signed graphs to be switching equivalent is given by Zaslavsky.

\begin{lemma}(\cite{Zaslavsky1982-1})\label{SE-lamma}
Two signed graphs $(G,\sigma_1)$ and $(G,\sigma_2)$ are switching equivalent if and only if they have the same set of negative cycles. 
\end{lemma}

The following lemma is a direct consequence of Lemma~\ref{SE-lamma}.

\begin{lemma}
A signed graph $(G,\sigma)$ is balanced if and only if it is switching equivalent to $(G,+)$. 
\end{lemma}

Two signed graphs are {\it isomorphic} to each other if there exists a graph isomorphism between their underlying graphs preserving the edge signs. Two signed graphs are {\it switching isomorphic} to each other if one is isomorphic to a switching of other. A {\it switching isomorphism class} of $G$ is the collection of all signed graphs, having underlying graph $G$, in which any two signed graphs are either switching equivalent or switching isomorphic to each other. 

 An {\it incidence} of $G$ is a pair $(v,e)$, where $v$ is one of the end vertices of an edge $e \in E(G)$. Thus corresponding to every edge of $G$ there are two incidences. The set of all incidences of $G$ is denoted by $I(G)$. Behr~\cite{Behr2020} defined edge coloring of signed graphs in terms of incidences (rather than just edges themselves) in order to incorporate edge signs and to make the definition compatible with switching operation. The definition of edge coloring and chromatic index of a signed graph, given by Behr, is as follows.

\begin{definition}({\cite{Behr2020}})\label{Definition - edge coloring of SGs}
\rm{A $q${\it -edge coloring} $\gamma$ of $\Sigma$ is an assignment of colors from the set $M_q$ to each incidence of $\Sigma$ subject to the condition that $\gamma(v,e) = - \sigma(e) \gamma(w,e)$ for each edge $e=vw$, where $M_q= \{ \pm 1, \ldots , \pm r \}$ if $q=2r$ and $M_q= \{0,\pm 1, \ldots , \pm r\}$ if $q=2r+1$. A $q$-edge coloring is {\it proper} if for any two incidences $(v,e)$ and $(v,f)$, involving the same vertex, $\gamma(v,e) \neq  \gamma(v,f)$. The {\it chromatic index} of a signed graph $\Sigma$, denoted by $\chi'(\Sigma)$, is the minimum number $q$ for which $\Sigma$ has a proper $q$-edge coloring.}
\end{definition}

From Definition~\ref{Definition - edge coloring of SGs}, it is clear that if we assign a non-zero color $c$ to a negative edge $e$, then $e$ must receive the color $c$ at both incidences. However, if $e$ is positive, then one of the incidences of $e$ receives the color $c$ while the other incidence receives the color $-c$. If $0 \in M_q$, then for any positive or negative edge we can assign the color $0$ to both of its incidences. Thus, the chromatic index of a signed graph $(G,\sigma)$ depends on the underlying graph $G$ as well as on its signature $\sigma$. Janczewski et al. \cite{Janczewski2023} studied the graph $G$ for which $\chi'(G,\sigma)$ does not depend on $\sigma$. To achieve this goal, they introduced two new classes of graphs, namely $1^{\pm}$ and $2^{\pm}$, such that graph $G$ is class $1^{\pm}$ (respectively, $2^{\pm}$) if and only if $\chi'(G,\sigma) = \Delta (G)$ (respectively, $\chi'(G,\sigma) = \Delta (G) +1$) for all possible signatures $\sigma$.

Since negative edges receive the same color at both of their incidences, the following lemma is immediate.

\begin{lemma}(\cite{Behr2020})\label{Lemma-(G,-) and G have same index}
For all-negative signed graph $(G,-)$, there is a one-to-one correspondence between proper $q$-edge colorings of $(G,-)$ and proper $q$-edge colorings of $G$.
\end{lemma}

For any graph $G$, by Lemma~\ref{Lemma-(G,-) and G have same index}, it is obvious that $\chi' (G)=\chi'(G,-)$, where $(G,-)$ is all-negative graph over $G$. Behr~\cite{Behr2020} proved that the edge coloring of a signed graph is compatible with switching operation. That is, switching does not affect the chromatic index of a signed graph.

\begin{lemma}(\cite{Behr2020})\label{Lemma-invariance of chro index}
Suppose $\gamma$ is a proper $q$-edge coloring of $(G,\sigma)$ and suppose $(G,\sigma')$ is obtained from $(G,\sigma)$ by switching a vertex set $X$. If $\gamma'$ is a new coloring which is obtained from $\gamma$ by negating all colors on all incidences involving vertices from
$X$, then $\gamma'$ is a proper $q$-edge coloring of $(G,\sigma')$.
\end{lemma}

This lemma tells us that if $(G,\sigma)$ and $(G,\sigma')$ are switching equivalent and $\chi'(G,\sigma)=q$ then $\chi'(G,\sigma')=q$. 

Behr proved a signed graph version of Vizing's theorem which is stated as follows.

\begin{theorem}(\cite{Behr2020})\label{Vizing's thm}
For a signed graph $(G,\sigma)$, $\Delta(G) \leq \chi'(G,\sigma) \leq \Delta(G)+1.$
\end{theorem} 

A signed graph $(G,\sigma)$ is {\it class} 1 if $\chi'(G,\sigma) = \Delta(G)$ and {\it class} 2 if $\chi'(G,\sigma) = \Delta(G)+1.$ 

Let $\Sigma =(G,\sigma)$ be a signed graph and let $\gamma$ be a proper $q$-edge coloring of $\Sigma$. The subgraph whose edge are colored using $\pm c$ with respect to $\gamma$ is denoted by $\Sigma_{c}[\gamma]$. If we have only one coloring in mind, then we write $\Sigma_{c}$. We call $\Sigma_{c}$ the \textit{c-graph} of $\Sigma$ with respect to $\gamma$. Furthermore, the maximum degree of $\Sigma_{c}$ is two because at most $c$ and $-c$ are present at each vertex of $\Sigma_{c}$. Therefore, every component of $\Sigma_{c}$ is either a path or a cycle. If $c=0$, then the maximum degree of $\Sigma_{0}$ is one and so $\Sigma_{0}$ is a matching. In \cite{Behr2020}, author also classified signed paths and signed cycles that can be possibly appear in $\Sigma_{c}$ when $c \neq 0$. More precisely, we have the following two results.

\begin{theorem}(\cite{Behr2020})
\label{Theorem-paths in Sigma_a}
Every signed path can be properly edge colored with $\pm a$ (where $a \neq 0$). Furthermore, every signed path has
exactly two different $\pm a$ colorings.
\end{theorem}

\begin{theorem}(\cite{Behr2020})
\label{Theorem-cycles in Sigma_a}
A signed circle $C$ can be properly colored with $\pm a$ (where $a \neq 0$) if and only if $C$ is positive. Furthermore, every positive circle has exactly two $\pm a$ colorings.
\end{theorem}

\noindent
\textbf{Observation 1.} Let $\gamma$ be a proper edge coloring of $\Sigma$. Then by Theorem~\ref{Theorem-paths in Sigma_a}, Theorem~\ref{Theorem-cycles in Sigma_a} and above discussion, it follows that $\Sigma_a[\gamma]$ consists of paths or positive cycles so that $\Sigma_a[\gamma]$ is a balanced subgraph of maximum degree 2.

We will apply the concept of $c$-graph $\Sigma_{c}$ and Observation 1 in the computation of the chromatic index of signed complete graphs (in Subsection 3.2). A result about the chromatic index of signed cycles is the following:

\begin{prop}(\cite{Zhang2020})
\label{Prop-chro index of signed cycles}
For any signed cycle $(C,\sigma)$, $\chi ' (C, \sigma) = \begin{cases}
2,~\text{if}~(C, \sigma)~\text{is balanced};\\
3,~\text{otherwise}.
\end{cases}$
\end{prop}

\section{Main results}
\subsection{Signed generalized book graphs}
Throughout this subsection, we assume that $m,n,k$ are positive integers such that $m \geq 3$, $n \geq 2$, $k \geq 2$ and $m-k \geq 1$. Let the vertex set of $B(m,n,k)$ be $V(B(m,n,k)) = \{v_1,\ldots ,v_k\} \cup \{u_{j}^i~|~1 \leq i \leq n,~~ 1 \leq j \leq m-k \}$, and let $v_1v_2...v_k$ be the common path to the cycles $C_{m}^{i}$, where $C_{m}^{i} = v_1u_{1}^{i}u_{2}^{i}...u_{m-k}^{i}v_k...v_1$, for $1 \leq i \leq n$. For example, the cycle $C_{5}^{1}$ in $B(5,3,3)$ is the cycle $v_1u_{1}^{1}u_{2}^{1}v_3v_2v_1$, where the graph $B(5,3,3)$ is shown in Figure~\ref{BG1}.

\begin{figure}[ht]
\centering
\begin{tikzpicture}[scale=0.7]
\node[fill=black, circle, inner sep=1.5pt] (v1) at (9,2) {};
\node [left] at (9,2) {$v_3$};
\node[fill=black, circle, inner sep=1.5pt] (v) at (9,3.6) {};
\node [left] at (9,3.6) {$v_2$};
\node[fill=black, circle, inner sep=1.5pt] (v2) at (9,5.3) {};
\node [left] at (9,5.3) {$v_1$};
\node[fill=black, circle, inner sep=1.5pt] (v3) at (11.5,3.1) {};
\node [right] at (11.5,3.1) {$u_{2}^{3}$};
\node[fill=black, circle, inner sep=1.5pt] (v4) at (11.5,4.2) {};
\node [right] at (11.5,4.2) {$u_{1}^{3}$};
\node[fill=black, circle, inner sep=1.5pt] (v5) at (13,2) {};
\node [right] at (13,2) {$u_{2}^{2}$};
\node[fill=black, circle, inner sep=1.5pt] (v6) at (13,5.3) {};
\node [right] at (13,5.3) {$u_{1}^{2}$};
\node[fill=black, circle, inner sep=1.5pt] (v7) at (14.5,0.9) {};
\node [right] at (14.5,0.9) {$u_{2}^{1}$};
\node[fill=black, circle, inner sep=1.5pt] (v8) at (14.5,6.3) {};
\node [right] at (14.5,6.3) {$u_{1}^{1}$};

\foreach \from/\to in {v/v1,v/v2,v1/v2,v1/v3,v2/v4,v1/v5,v2/v6,v1/v7,v2/v8,v3/v4,v5/v6,v7/v8} \draw (\from) -- (\to);

\end{tikzpicture}
\caption{The generalized book graph $B(5,3,3)$.}
\label{BG1}
\end{figure}
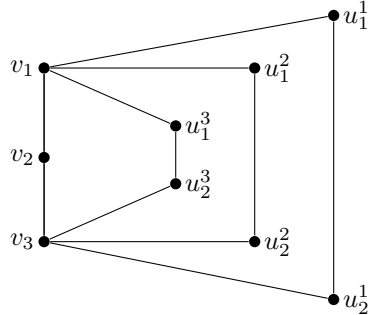

Since the chromatic index of a signed graph is invariant under switching operation and isomorphism, so to compute the chromatic index of all signed generalized book graphs it is enough to determine the chromatic index of switching non-isomorphic signed generalized book graphs. So we first determine the switching non-isomorphic signed generalized book graphs.

\begin{theorem}\label{theorem1}
Let $(B(m,n,k),\sigma)$ be a signed generalized book graph. Then $(B(m,n,k),\sigma)$ is equivalent to $(B(m,n,k),\tau)$, where $\tau^{-1}(-1) \subseteq \{v_1u_{1}^{1},\ldots,v_1u_{1}^{n}\}$. Furthermore, the number of switching non-isomorphic signed $B(m,n,k)$ is $n+1$.  
\end{theorem}
\begin{proof}
Let $(B(m,n,k),\sigma)$ be a signed generalized book graph. It is clear that every signed cycle $(C_m, \pi)$ is switching equivalent to a signed cycle $(C_m, \pi')$, where the number of negative edges in $(C_n, \pi')$ is at most one. Thus by suitable switchings, if needed, we get $(B(m,n,k),\sigma')$ equivalent to $(B(m,n,k),\sigma)$ so that every negative edge of $(B(m,n,k),\sigma')$ is incident to $v_1$. Further, if the edge $v_1v_2$ is negative in  $(B(m,n,k),\sigma')$, switching $v_1$ will make it positive. Thus we get a signed generalized book graph $(B(m,n,k),\tau)$ equivalent to $(B(m,n,k),\sigma)$, where $\tau^{-1}(-1) \subseteq \{v_1u_{1}^{1},\ldots,v_1u_{1}^{n}\}$. This proves the first part of theorem.

Now let $(B(m,n,k),\sigma)$ and $(B(m,n,k),\tau)$ be any two signed generalized book graphs. By part (i), without loss of generality, we can assume $\sigma^{-1}(-1),\tau^{-1}(-1) \subseteq \{v_1u_{1}^{1},\ldots,v_1u_{1}^{n}\}$. If $|\sigma^{-1}(-1)|=|\tau^{-1}(-1)|$, then an one-one correspondence between $\sigma$ and $\tau$ determines an isomorphism between \linebreak $(B(m,n,k),\sigma)$ and $(B(m,n,k),\tau)$. If $|\sigma^{-1}(-1)|\neq |\tau^{-1}(-1)|$, then $(B(m,n,k),\sigma)$ cannot be switching isomorphic to $(B(m,n,k),\tau)$ because both signed graphs have different number of negative cycles $C_m$.  Therefore, the number of switching non-isomorphic signed $B(m,n,k)$ is $n+1$.
\end{proof}

For each $1 \leq l \leq n$, let $\sigma_{l}^{-1}(-1) = \{v_1u_{1}^{1}, \ldots , v_1u_{1}^{l}\}$. If $\sigma_{0}^{-1}(-1) =\emptyset$, then by the preceding theorem, $\sigma_0,\sigma_1,\ldots, \sigma_n$ are switching non-isomorphic signatures of $B(m,n,k)$. This means, each $(B(m,n,k),\sigma_{l})$ is a representative of $n+1$ switching isomorphism classes of $B(m,n,k)$, where $l=0,1,\ldots , n$. Two switching non-isomorphic signed generalized book graphs $(B(5,3,3),\sigma_1)$ and $(B(5,3,3),\sigma_2)$ are shown in Figure~\ref{Fig-SNI graphs}.

\begin{figure}[ht]
\hspace*{0.5cm}
\begin{subfigure}{0.4\textwidth}
\begin{tikzpicture}[scale=0.5]
\node[fill=black, circle, inner sep=1.5pt] (v1) at (9,2) {};
\node [left] at (9,2) {$v_3$};
\node[fill=black, circle, inner sep=1.5pt] (v) at (9,3.6) {};
\node [left] at (9,3.6) {$v_2$};
\node[fill=black, circle, inner sep=1.5pt] (v2) at (9,5.3) {};
\node [left] at (9,5.3) {$v_1$};
\node[fill=black, circle, inner sep=1.5pt] (v3) at (11.5,3.1) {};
\node [right] at (11.5,3.1) {$u_{2}^{3}$};
\node[fill=black, circle, inner sep=1.5pt] (v4) at (11.5,4.2) {};
\node [right] at (11.5,4.2) {$u_{1}^{3}$};
\node[fill=black, circle, inner sep=1.5pt] (v5) at (13,2) {};
\node [right] at (13,2) {$u_{2}^{2}$};
\node[fill=black, circle, inner sep=1.5pt] (v6) at (13,5.3) {};
\node [right] at (13,5.3) {$u_{1}^{2}$};
\node[fill=black, circle, inner sep=1.5pt] (v7) at (14.5,0.9) {};
\node [right] at (14.5,0.9) {$u_{2}^{1}$};
\node[fill=black, circle, inner sep=1.5pt] (v8) at (14.5,6.3) {};
\node [right] at (14.5,6.3) {$u_{1}^{1}$};

\foreach \from/\to in {v/v1,v/v2,v1/v2,v1/v3,v2/v4,v1/v5,v2/v6,v1/v7,v3/v4,v5/v6,v7/v8} \draw (\from) -- (\to);
\draw [dashed] (9,5.3) -- (14.5,6.3);

\end{tikzpicture}
\caption{The signed graph $(B(5,3,3),\sigma_1)$.} \label{Fig-signed graph B(5,3,3)}
\end{subfigure}
\hfill
\hspace*{-1cm}
\begin{subfigure}{0.4\textwidth}
\begin{tikzpicture}[scale=0.5]
\node[fill=black, circle, inner sep=1.5pt] (v1) at (9,2) {};
\node [left] at (9,2) {$v_3$};
\node[fill=black, circle, inner sep=1.5pt] (v) at (9,3.6) {};
\node [left] at (9,3.6) {$v_2$};
\node[fill=black, circle, inner sep=1.5pt] (v2) at (9,5.3) {};
\node [left] at (9,5.3) {$v_1$};
\node[fill=black, circle, inner sep=1.5pt] (v3) at (11.5,3.1) {};
\node [right] at (11.5,3.1) {$u_{2}^{3}$};
\node[fill=black, circle, inner sep=1.5pt] (v4) at (11.5,4.2) {};
\node [right] at (11.5,4.2) {$u_{1}^{3}$};
\node[fill=black, circle, inner sep=1.5pt] (v5) at (13,2) {};
\node [right] at (13,2) {$u_{2}^{2}$};
\node[fill=black, circle, inner sep=1.5pt] (v6) at (13,5.3) {};
\node [right] at (13,5.3) {$u_{1}^{2}$};
\node[fill=black, circle, inner sep=1.5pt] (v7) at (14.5,0.9) {};
\node [right] at (14.5,0.9) {$u_{2}^{1}$};
\node[fill=black, circle, inner sep=1.5pt] (v8) at (14.5,6.3) {};
\node [right] at (14.5,6.3) {$u_{1}^{1}$};

\foreach \from/\to in {v/v1,v/v2,v1/v2,v1/v3,v2/v4,v1/v5,v1/v7,v3/v4,v5/v6,v7/v8} \draw (\from) -- (\to);
\draw [dashed] (9,5.3) -- (14.5,6.3);
\draw [dashed] (9,5.3) -- (13,5.3);

\end{tikzpicture}
\caption{The signed graph $(B(5,3,3),\sigma_2)$.}\label{Fig- a signed B(5,3,3)}
\end{subfigure}
\caption{Two switching non-isomorphic signed generalized book graphs over $B(5,3,3)$. Throughout, dashed lines represent negative edges and solid lines represent positive edges.}\label{Fig-SNI graphs}
\end{figure}
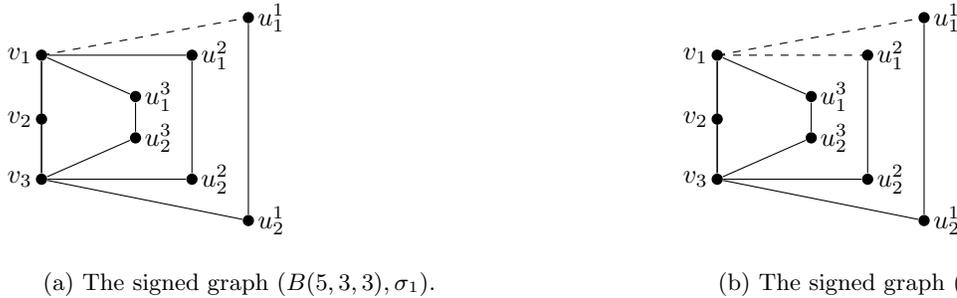

Now we compute the value of $\chi'(B(m,n,k),\sigma_{l})$ for $l=0,1,\ldots , n$.

\begin{theorem}\label{Theorem-chrom index of unsigned GBG}
For signature $\sigma_0$, $\chi'(B(m,n,k),\sigma_{0}) = n+1$.
\end{theorem}
\begin{proof}
We discuss two cases separately.

\noindent
\textbf{Case 1.} Suppose $n=2r$ for some integer $r \geq 1$. In this case, we have that $\Delta(B(m,2r,k))=2r+1$. Thus, $\chi' (B(m,2r,k),\sigma_{0}) \geq 2r+1$ by Theorem~\ref{Vizing's thm}. Now we give a proper $(2r+1)$-edge coloring $\gamma$ of $(B(m,2r,k),\sigma_{0})$ as follows (see Figure~\ref{Fig-coloring of B(4,2,3)} for the case $m=4,n=2,k=3$). 
\begin{itemize}
\item[1.] We sequentially color the incidences of the path $v_1v_2 \cdots v_{k}$ using the pattern $((-1)(1)) \ldots ((-1)(1))$.
\item[2.] We sequentially color the incidences of path $v_1u_{1}^{1} \cdots u_{m-k}^{1}$ using the pattern $((1)(-1)) \ldots ((1)(-1))$. For the edge $u_{m-k}^{1}v_k$, we set $\gamma (u_{m-k}^{1},u_{m-k}^{1}v_k) = \gamma (v_k,v_ku_{m-k}^{1}) =0$.

\item[3.] For the edge $v_1u_{1}^{2}$, we set $\gamma (v_1,v_1u_{1}^{2}) = \gamma (u_{1}^{2},u_{1}^{2}v_1) =0$. Now we sequentially color the incidences of the path $u_{1}^{2} \cdots u_{m-k}^{2}v_k$ using the pattern $((1)(-1)) \ldots ((1)(-1))$.

\item[4.] Finally, for $ i = 3,\ldots , 2r$, we sequentially color the incidences of the path $v_1u_{1}^{i} \cdots u_{m-k}^{i}v_k$ using the pattern $((\frac{i+1}{2})(-\frac{i+1}{2})) \ldots ((\frac{i+1}{2})(-\frac{i+1}{2}))$ and $((-\frac{i}{2})(\frac{i}{2})) \ldots ((-\frac{i}{2})(\frac{i}{2}))$ according as $i$ is odd and even, respectively. (This step is needed only when $r \geq 2$.)
\end{itemize}

The so-obtained coloring $\gamma$ is clearly a proper edge coloring of $(B(m,2r,k),\sigma_{0})$ for all $m \geq 3$ and $k \geq 2$.

\noindent
\textbf{Case 2.} Suppose $n=2r-1$ for some integer $r \geq 2$. In this case, we have that $\Delta(B(m,2r-1,k))=2r$. Thus, $\chi' (B(m,2r-1,k),\sigma_{0}) \geq 2r$ by Theorem~\ref{Vizing's thm}. Now we give a proper $(2r)$-edge coloring $\gamma$ of $(B(m,2r-1,k),\sigma_{0})$ as follows (see Figure~\ref{Fig-coloring of B(5,3,3)} for the case $m=5,n=3,k=3$). 

\begin{itemize}
\item[1.] We sequentially color the incidences of the path $v_1v_2 \cdots v_{k}$ using the pattern $((-r)(r)) \ldots ((-r)(r))$.
\item[2.] For $ i = 1,\ldots , 2r-1$, we sequentially color the incidences of the path $v_1u_{1}^{i} \cdots u_{m-k}^{i}v_k$ using the pattern $((\frac{i+1}{2})(-\frac{i+1}{2})) \ldots ((\frac{i+1}{2})(-\frac{i+1}{2}))$ and $((-\frac{i}{2})(\frac{i}{2})) \ldots ((-\frac{i}{2})(\frac{i}{2}))$ according as $i$ is odd and even, respectively.
\end{itemize}

The so-obtained coloring $\gamma$ is clearly a proper edge coloring of $(B(m,2r-1,k),\sigma_{0})$ for all $m \geq 3$ and $k \geq 2$.

From Cases 1 and 2, we conclude that $\chi'(B(m,n,k),\sigma_{0}) = n+1$ for all $m \geq 3$, $n \geq 2$ and $k \geq 2$.
\end{proof}

\begin{figure}[ht]
\hspace*{0.5cm}
\begin{subfigure}{0.3\textwidth}
\begin{tikzpicture}[scale=0.7]
\node[fill=black, circle, inner sep=1.5pt] (v1) at (9,0) {};
\node[fill=black, circle, inner sep=1.5pt] (v2) at (9,3) {};
\node[fill=black, circle, inner sep=1.5pt] (v3) at (9,6) {};
\node[fill=black, circle, inner sep=1.5pt] (v4) at (10.5,3) {};
\node[fill=black, circle, inner sep=1.5pt] (v5) at (13.5,3) {};

\foreach \from/\to in {v1/v2,v2/v3,v1/v4,v1/v5,v3/v4,v3/v5} \draw (\from) -- (\to);

\node [right] at (9.8,5.6) {$1$};
\node [right] at (9.5,4.9) {$0$};

\node [left] at (9,5.5) {-$1$};
\node [left] at (9,3.5) {$1$};
\node [left] at (9,2.5) {-$1$};
\node [left] at (9,0.5) {$1$};

\node [above] at (10.6,3.2) {$0$};
\node [above] at (13.3,3.3) {-$1$};

\node [below] at (10.5,2.8) {$1$};
\node [below] at (13.3,2.9) {$0$};

\node [right] at (9.8,0.4) {$0$};
\node [right] at (9.4,1) {-$1$};

\end{tikzpicture}
\caption{Proper 3-edge coloring of $(B(4,2,3), \sigma_0)$.} \label{Fig-coloring of B(4,2,3)}
\end{subfigure}
\hfill
\hspace*{-1cm}
\begin{subfigure}{0.3\textwidth}
\begin{tikzpicture}[scale=0.6]
\node[fill=black, circle, inner sep=1.5pt] (v1) at (9,0) {};
\node[fill=black, circle, inner sep=1.5pt] (v) at (9,3) {};
\node[fill=black, circle, inner sep=1.5pt] (v2) at (9,6) {};
\node[fill=black, circle, inner sep=1.5pt] (v3) at (10.5,2) {};
\node[fill=black, circle, inner sep=1.5pt] (v4) at (10.5,4) {};
\node[fill=black, circle, inner sep=1.5pt] (v5) at (12,1) {};
\node[fill=black, circle, inner sep=1.5pt] (v6) at (12,5) {};
\node[fill=black, circle, inner sep=1.5pt] (v7) at (13.5,-0.3) {};
\node[fill=black, circle, inner sep=1.5pt] (v8) at (13.5,6.3) {};

\foreach \from/\to in {v/v1,v/v2,v1/v2,v1/v3,v2/v4,v1/v5,v2/v6,v1/v7,v2/v8,v3/v4,v5/v6,v7/v8} \draw (\from) -- (\to);


\node [above] at (9.6,6) {$1$};
\node [below] at (10,5.65) {-$1$};
\node [below] at (9.3,5.45) {$2$};

\node [left] at (9,5.5) {-$2$};
\node [left] at (9,3.5) {$2$};
\node [left] at (9,2.5) {-$2$};
\node [left] at (9,0.5) {$2$};

\node [above] at (13,6.2) {-$1$};
\node [below] at (11.45,5.25) {$1$};
\node [below] at (9.8,4.7) {-$2$};

\node [below] at (9.6,0) {-$1$};
\node [above] at (9.8,0.2) {$1$};
\node [above] at (9.4,0.65) {-$2$};

\node [below] at (13,-0.2) {$1$};
\node [above] at (11.45,0.75) {-$1$};
\node [above] at (10,1.4) {$2$};

\node [right] at (13.5,5.8) {$1$};
\node [right] at (12,4.5) {-$1$};
\node [right] at (10.5,3.5) {$2$};

\node [right] at (13.5,0.2) {-$1$};
\node [right] at (12,1.5) {$1$};
\node [right] at (10.5,2.5) {-$2$};

\end{tikzpicture}
\caption{Proper 4-edge coloring of $(B(5,3,3), \sigma_0)$.}\label{Fig-coloring of B(5,3,3)}
\end{subfigure}
\caption{Proper edge colorings of all-positive graphs over $B(4,2,3)$ and $B(5,3,3)$.}\label{Fig-edge colorings of B(4,2,3) and B(5,3,3)}
\end{figure}
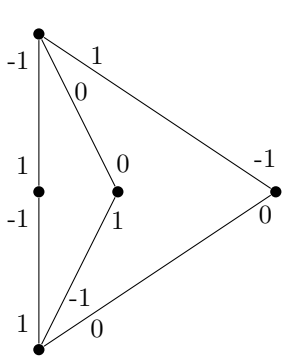
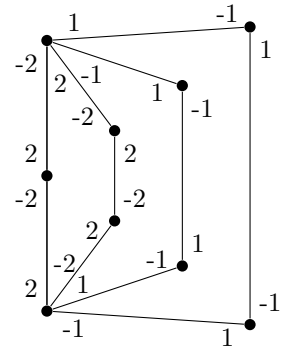

Now we determine the value of $\chi'(B(m,n,k),\sigma_{n})$ for all $m \geq 3$, $n \geq 2$ and $k \geq 2$.

By Lemma~\ref{Lemma-(G,-) and G have same index}, it is obvious that $\chi'(B(m,n,k)) =\chi'(B(m,n,k), -)$. But it is important to note that $(B(m,n,k), -)$ is switching equivalent to $(B(m,n,k), +)$ or $(B(m,n,k), \{v_1v_2\})$ according as $m$ is even or odd, respectively. Furthermore, by switching $v_1$ of $(B(m,n,k), \{v_1v_2\})$ we get $(B(m,n,k), \sigma_{n})$. These facts imply that $(B(m,n,k), -)$ is equivalent to $(B(m,n,k), \sigma_{n})$ for odd values of $m$. Thus we have the following theorem which directly follows from Lemma~\ref{Lemma-(G,-) and G have same index}, Lemma~\ref{Lemma-invariance of chro index} and Theorem~\ref{Theorem-chrom index of unsigned GBG}.

\begin{theorem}
For any odd $m \geq 3$, $\chi'(B(m,n,k),\sigma_{n}) = n+1$.
\end{theorem} 

It remains to compute the value $\chi'(B(m,n,k),\sigma_{n})$ for even values of $m$. We do it in the next result.

\begin{theorem}
For any even $m \geq 4$, $\chi'(B(m,n,k),\sigma_{n}) = n+1$.
\end{theorem}
\begin{proof}
By switching $v_1$ in $(B(m,n,k),\sigma_{n})$, we get signed generalized book graph $(B(m,n,k), \{v_1v_2\})$. Thus, due to Lemma~\ref{Lemma-invariance of chro index}, instead of finding the value of $\chi'(B(m,n,k),\sigma_{n})$, we will find the value of $\chi'(B(m,n,k), \{v_1v_2\})$. Obviously, $\chi'(B(m,n,k), \{v_1v_2\}) \geq n+1$ since $\Delta ((B(m,n,k)) = n+1$. 
We discuss two cases separately.

\noindent
\textbf{Case 1.} Suppose $n=2r$ for some integer $r \geq 1$. In this case, we have that $\Delta(B(m,2r,k))=2r+1$. Thus, $\chi' (B(m,2r,k), \{v_1v_2\}) \geq 2r+1$ by Theorem~\ref{Vizing's thm}. Now we give a proper $(2r+1)$-edge coloring $\gamma$ of $(B(m,2r,k), \{v_1v_2\})$.

\textbf{Subcase 1.1.} If $k = 2$, then $\gamma$ is defined as follows (see Figure~\ref{Fig-coloring of B(4,2,2)} for the case $m=4,n=2,k=2$). 
\begin{itemize}
\item[1.] For the edge $v_1v_2$, we set $\gamma (v_1,v_1v_2) = \gamma (v_2,v_2v_1) =0$.
\item[2.] For $ i = 1,\ldots , 2r$, we sequentially color the incidences of the path $v_1u_{1}^{i} \cdots u_{m-k}^{i}v_k$ using the pattern $((\frac{i+1}{2})(-\frac{i+1}{2})) \ldots ((\frac{i+1}{2})(-\frac{i+1}{2}))$ and $((-\frac{i}{2})(\frac{i}{2})) \ldots ((-\frac{i}{2})(\frac{i}{2}))$ according as $i$ is odd and even, respectively. 
\end{itemize}

\textbf{Subcase 1.2.} If $k \geq 3$, then $\gamma$ is defined as follows (see Figure~\ref{Fig-coloring of B(6,2,3)} for the case $m=6,n=2,k=3$). 
\begin{itemize}
\item[1.] For the edge $v_1v_2$, we set $\gamma (v_1,v_1v_2) = \gamma (v_2,v_2v_1) =0$. 
\item[2.] For $j=2,\ldots , k-1$, incidences $(v_{j},v_{j}v_{j+1})$ and $(v_{j+1},v_{j+1}v_{j})$ are colored with 1 and $-1$, respectively.
\item[3.] We sequentially color the incidences of path $v_1u_{1}^{1} \cdots u_{m-k}^{1}$ using the pattern $((1)(-1)) \ldots ((1)(-1))$. For the edge $u_{m-k}^{1}v_k$, we set $\gamma (u_{m-k}^{1},u_{m-k}^{1}v_k) = \gamma (v_k,v_ku_{m-k}^{1}) =0$.
\item[4.] For $ i = 2,\ldots , 2r$, we sequentially color the incidences of the path $v_1u_{1}^{i} \cdots u_{m-k}^{i}v_k$ using the pattern $((-\frac{i}{2})(\frac{i}{2})) \ldots ((-\frac{i}{2})(\frac{i}{2}))$ and $((\frac{i+1}{2})(-\frac{i+1}{2})) \ldots ((\frac{i+1}{2})(-\frac{i+1}{2}))$  according as $i$ is even and odd, respectively.
\end{itemize} 

The so-obtained coloring $\gamma$ is clearly a proper edge coloring of $(B(m,2r,k),\{v_1v_2\})$ for all $m \geq 3$ and $k \geq 2$.

\noindent
\textbf{Case 2.} Suppose $n=2r-1$ for some integer $r \geq 2$. In this case, we have that $\Delta(B(m,2r-1,k))=2r$. Thus, $\chi' (B(m,2r-1,k), \{v_1v_2\}) \geq 2r$ by Theorem~\ref{Vizing's thm}. Now we give a proper $(2r)$-edge coloring $\gamma$ of $(B(m,2r-1,k), \{v_1v_2\})$ as follows (see Figure~\ref{Fig-coloring of B(6,3,4)} for the case $m=6,n=3,k=4$). 
\begin{itemize}
\item[1.] For the edge $v_1v_2$, we set $\gamma (v_1,v_1v_2) = \gamma (v_2,v_2v_1) =-r$ and if $k \geq 3$, then we color the incidences of $v_2v_3 \ldots v_{k-1}v_k$ using the pattern $((r)(-r)) \ldots ((r)(-r))$.
\item[2.] We sequentially color the incidences of path $v_1u_{1}^{1} \cdots u_{m-k}^{1}$ using the pattern $((1)(-1)) \ldots ((1)(-1))$. For the edge $u_{m-k}^{1}v_k$, we set $\gamma (u_{m-k}^{1},u_{m-k}^{1}v_k) = -r$ and $\gamma (v_k,v_ku_{m-k}^{1}) =r$. 
\item[3.] For $ i = 2,\ldots , 2r-2$, we sequentially color the incidences of the path $v_1u_{1}^{i} \cdots u_{m-k}^{i}v_k$ using the pattern $((-\frac{i}{2})(\frac{i}{2})) \ldots ((-\frac{i}{2})(\frac{i}{2}))$ and $((\frac{i+1}{2})(-\frac{i+1}{2})) \ldots ((\frac{i+1}{2})(-\frac{i+1}{2}))$ according as $i$ is even and odd, respectively.
\item[4.] We sequentially color the incidences of the path $v_1u_{1}^{2r-1} \cdots u_{m-k}^{2r-1}$ using the pattern \linebreak $((r)(-r)) \ldots ((r)(-r))$. For the edge $u_{m-k}^{2r-1}v_k$, we set $\gamma (u_{m-k}^{2r-1},u_{m-k}^{2r-1}v_k) = 1$ and $\gamma (v_k,v_ku_{m-k}^{2r-1}) = -1$. 
\end{itemize}
The so-obtained coloring $\gamma$ is clearly a proper edge coloring of $(B(m,2r-1,k),\{v_1v_2\})$ for all even $m \geq 4$ and $k \geq 2$. This completes the proof.
\end{proof}

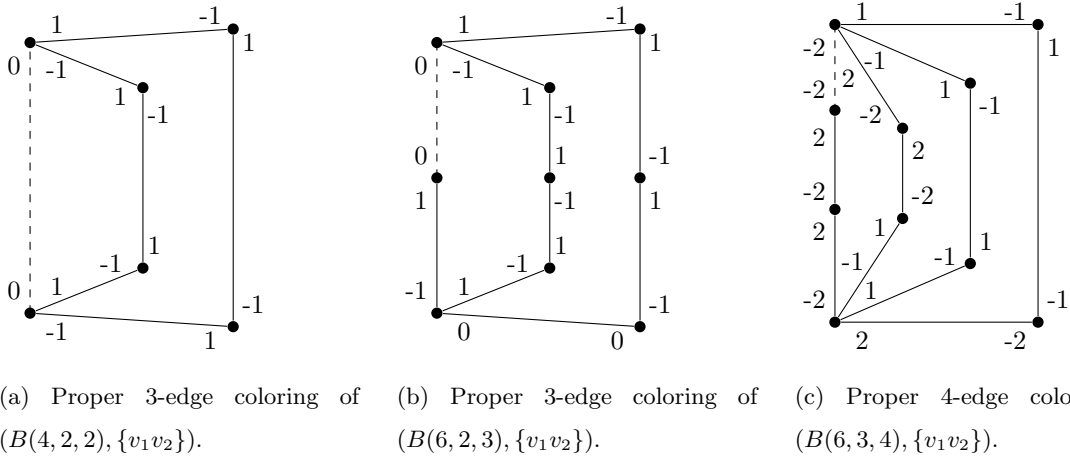
\begin{figure}[ht]
\hspace*{0.5cm}
\begin{subfigure}{0.3\textwidth}
\begin{tikzpicture}[scale=0.6]
\node[fill=black, circle, inner sep=1.5pt] (v1) at (9,0) {};
\node[fill=black, circle, inner sep=1.5pt] (v2) at (9,6) {};
\node[fill=black, circle, inner sep=1.5pt] (v3) at (11.5,1) {};
\node[fill=black, circle, inner sep=1.5pt] (v4) at (11.5,5) {};
\node[fill=black, circle, inner sep=1.5pt] (v5) at (13.5,-0.3) {};
\node[fill=black, circle, inner sep=1.5pt] (v6) at (13.5,6.3) {};

\foreach \from/\to in {v1/v3,v2/v4,v1/v5,v2/v6,v3/v4,v5/v6} \draw (\from) -- (\to);
\draw [dashed] (9,0) -- (9,6);

\node [left] at (9,5.5) {$0$};
\node [left] at (9,0.5) {$0$};

\node [above] at (9.6,6) {$1$};
\node [above] at (13,6.2) {-$1$};
\node [right] at (13.5,6) {$1$};
\node [right] at (13.5,0.2) {-$1$};
\node [below] at (13,-0.2) {$1$};
\node [below] at (9.6,0) {-$1$};

\node [above] at (9.6,5) {-$1$};
\node [below] at (11,5.2) {$1$};
\node [right] at (11.4,4.4) {-$1$};
\node [right] at (11.4,1.5) {$1$};
\node [above] at (10.8,0.7) {-$1$};
\node [above] at (9.6,0.2) {$1$};

\end{tikzpicture}
\caption{Proper 3-edge coloring of $(B(4,2,2),\{v_1v_2\})$.} \label{Fig-coloring of B(4,2,2)}
\end{subfigure}
\hfill
\hspace*{-1cm}
\begin{subfigure}{0.3\textwidth}
\begin{tikzpicture}[scale=0.6]
\node[fill=black, circle, inner sep=1.5pt] (v1) at (9,0) {};
\node[fill=black, circle, inner sep=1.5pt] (v2) at (9,3) {};
\node[fill=black, circle, inner sep=1.5pt] (v3) at (9,6) {};
\node[fill=black, circle, inner sep=1.5pt] (v4) at (11.5,1) {};
\node[fill=black, circle, inner sep=1.5pt] (v5) at (11.5,3) {};
\node[fill=black, circle, inner sep=1.5pt] (v6) at (11.5,5) {};
\node[fill=black, circle, inner sep=1.5pt] (v7) at (13.5,-0.3) {};
\node[fill=black, circle, inner sep=1.5pt] (v8) at (13.5,3) {};
\node[fill=black, circle, inner sep=1.5pt] (v9) at (13.5,6.3) {};

\foreach \from/\to in {v1/v2,v1/v4,v1/v7,v3/v6,v3/v9,v4/v5,v5/v6,v7/v8,v8/v9} \draw (\from) -- (\to);
\draw [dashed] (9,6) -- (9,3);

\node [left] at (9,5.5) {$0$};
\node [left] at (9,3.5) {$0$};
\node [left] at (9,2.5) {$1$};
\node [left] at (9,0.5) {-$1$};

\node [above] at (9.6,6) {$1$};
\node [above] at (13,6.2) {-$1$};
\node [right] at (13.5,6) {$1$};
\node [right] at (13.5,3.5) {-$1$};
\node [right] at (13.5,2.5) {$1$};
\node [right] at (13.5,0.2) {-$1$};
\node [below] at (13,-0.2) {$0$};
\node [below] at (9.6,0) {$0$};

\node [above] at (9.6,5) {-$1$};
\node [below] at (11,5.2) {$1$};
\node [right] at (11.4,4.4) {-$1$};
\node [right] at (11.4,3.5) {$1$};
\node [right] at (11.4,2.5) {-$1$};
\node [right] at (11.4,1.5) {$1$};
\node [above] at (10.8,0.7) {-$1$};
\node [above] at (9.6,0.2) {$1$};

\end{tikzpicture}
\caption{Proper 3-edge coloring of $(B(6,2,3),\{v_1v_2\})$.}\label{Fig-coloring of B(6,2,3)}
\end{subfigure}
\hfill
\hspace*{-1cm}
\begin{subfigure}{0.3\textwidth}
\begin{tikzpicture}[scale=0.6]
\node[fill=black, circle, inner sep=1.5pt] (v1) at (9,-0.3) {};
\node[fill=black, circle, inner sep=1.5pt] (v2) at (9,2.2) {};
\node[fill=black, circle, inner sep=1.5pt] (v3) at (9,4.4) {};
\node[fill=black, circle, inner sep=1.5pt] (v4) at (9,6.3) {};
\node[fill=black, circle, inner sep=1.5pt] (v5) at (10.5,2) {};
\node[fill=black, circle, inner sep=1.5pt] (v6) at (10.5,4) {};
\node[fill=black, circle, inner sep=1.5pt] (v7) at (12,1) {};
\node[fill=black, circle, inner sep=1.5pt] (v8) at (12,5) {};
\node[fill=black, circle, inner sep=1.5pt] (v9) at (13.5,-0.3) {};
\node[fill=black, circle, inner sep=1.5pt] (v10) at (13.5,6.3) {};

\foreach \from/\to in {v1/v2,v2/v3,v1/v5,v1/v7,v1/v9,v4/v6,v4/v8,v4/v10,v5/v6,v7/v8,v9/v10} \draw (\from) -- (\to);
\draw [dashed] (9,6.3) -- (9,4.4);


\node [left] at (9,5.8) {-$2$};
\node [left] at (9,4.85) {-$2$};
\node [left] at (9,3.8) {$2$};
\node [left] at (9,2.6) {-$2$};
\node [left] at (9,1.7) {$2$};
\node [left] at (9,0.2) {-$2$};

\node [above] at (9.6,6.2) {$1$};
\node [above] at (13,6.2) {-$1$};
\node [right] at (13.5,5.8) {$1$};
\node [right] at (13.5,0.2) {-$1$};
\node [below] at (13,-0.3) {-$2$};
\node [below] at (9.6,-0.3) {$2$};

\node [below] at (9.9,5.9) {-$1$};
\node [below] at (11.45,5.25) {$1$};
\node [right] at (12,4.5) {-$1$};
\node [right] at (12,1.5) {$1$};
\node [above] at (11.45,0.75) {-$1$};
\node [above] at (9.8,0) {$1$};

\node [below] at (9.3,5.5) {$2$};
\node [below] at (9.8,4.7) {-$2$};
\node [right] at (10.5,3.5) {$2$};
\node [right] at (10.5,2.5) {-$2$};
\node [above] at (10,1.4) {$1$};
\node [above] at (9.4,0.65) {-$1$};

\end{tikzpicture}
\caption{Proper 4-edge coloring of $(B(6,3,4),\{v_1v_2\})$.}\label{Fig-coloring of B(6,3,4)}
\end{subfigure}
\caption{Proper edge colorings of some signed generalized book graphs.}\label{Fig-edge colorings of B(4,2,2), B(6,2,3) and B(6,3,4)}
\end{figure}

Now, we compute the chromatic index of $\chi'(B(m,n,k),\sigma_{1})$, where $\sigma_{1} ^{-1}(-1) = \{v_1u_{1}^{1}\}$.

\begin{theorem}
Let $m \geq 3, n \geq 2, k \geq 2$, then $\chi'(B(m,n,k),\sigma_{1}) = n+1$.
\end{theorem}
\begin{proof}
Recall that the set of negative edges, in a signed generalized book graph with signature $\sigma_1$, is $\{v_1u_{1}^{1}\}$. We distinguish two cases according to the parity of $n$.

\noindent
\textbf{Case 1.} Suppose $n=2r$ for some integer $r \geq 1$. Thus, $\chi' (B(m,2r,k), \{v_1u_{1}^{1}\}) \geq 2r+1$ by Theorem~\ref{Vizing's thm}. Now we give a proper $(2r+1)$-edge coloring $\gamma$ of $(B(m,2r,k), \{v_1u_{1}^{1}\})$ as follows (see Figure~\ref{Fig-coloring of B(6,2,3) with sigma_1} for the case $m=6,n=2,k=3$). 
\begin{itemize}
\item[1.] We sequentially color the incidences of the path $v_1 \ldots v_k$ using the pattern $((-r)(r)) \ldots ((-r)(r))$. 
\item[2.] For the edge $v_1u_{1}^{1}$, we set $\gamma (v_1,v_1u_{1}^{1}) = \gamma (u_{1}^{1},u_{1}^{1}v_1) =0$ and we sequentially color the incidences of the path $u_{1}^{1} \ldots u_{m-k}^{1} v_k$ using the pattern $((r)(-r)) \ldots ((r)(-r))$.  
\item[3.] For $ i = 2,\ldots , 2r-1$, we sequentially color the incidences of the path $v_1u_{1}^{i} \cdots u_{m-k}^{i}v_k$ using the pattern $((\frac{i}{2})(-\frac{i}{2})) \ldots ((\frac{i}{2})(-\frac{i}{2}))$ and $((-\frac{i-1}{2})(\frac{i-1}{2})) \ldots ((-\frac{i-1}{2})(\frac{i-1}{2}))$  according as $i$ is even and odd, respectively.
\item[4.] We sequentially color the incidences of $v_1u_{1}^{2r} \ldots u_{m-k}^{2r}$ using the pattern $((r)(-r)) \ldots ((r)(-r))$ and for the edge $u_{m-k}^{2r}v_k$, we set $\gamma (u_{m-k}^{2r},u_{m-k}^{2r}v_k) = \gamma (v_k,v_ku_{m-k}^{2r}) =0$. 
\end{itemize} 
The so-obtained coloring $\gamma$ is clearly a proper edge coloring of $(B(m,2r,k),\{v_1u_{1}^{1}\})$ for all $m \geq 3$ and $k \geq 2$.

\noindent
\textbf{Case 2.} Suppose $n=2r-1$ for some integer $r \geq 2$. Thus, $\chi' (B(m,2r-1,k), \{v_1u_{1}^{1}\}) \geq 2r$ by Theorem~\ref{Vizing's thm}. Now we give a proper $(2r)$-edge coloring $\gamma$ of $(B(m,2r-1,k), \{v_1u_{1}^{1}\})$ as follows (see Figure~\ref{Fig-coloring of B(6,3,4) with sigma_1} for the case $m=6,n=3,k=4$).
\begin{itemize}
\item[1.] We sequentially color the incidences of the path $v_1 \ldots v_k$ using the pattern $((-r)(r)) \ldots ((-r)(r))$. 
\item[2.] For the edge $v_1u_{1}^{1}$, we set $\gamma (v_1,v_1u_{1}^{1}) = \gamma (u_{1}^{1},u_{1}^{1}v_1) =1$ and we sequentially color the incidences of the path $u_{1}^{1} \ldots u_{m-k}^{1} v_k$ using the pattern $((r)(-r)) \ldots ((r)(-r))$.  
\item[3.] For $ i = 2,\ldots , 2r-2$, we sequentially color the incidences of the path $v_1u_{1}^{i} \cdots u_{m-k}^{i}v_k$ using the pattern $((-\frac{i}{2})(\frac{i}{2})) \ldots ((-\frac{i}{2})(\frac{i}{2}))$ and $((\frac{i+1}{2})(-\frac{i+1}{2})) \ldots ((\frac{i+1}{2})(-\frac{i+1}{2}))$  according as $i$ is even and odd, respectively.
\item[4.] We sequentially color the incidences of $v_1u_{1}^{2r-1} \ldots u_{m-k}^{2r-1}$ using the pattern $((r)(-r)) \ldots ((r)(-r))$ and for the edge $u_{m-k}^{2r-1}v_k$, we set $\gamma (u_{m-k}^{2r-1},u_{m-k}^{2r-1}v_k) =1$, $\gamma (v_k,v_ku_{m-k}^{2r-1}) =-1$. 
\end{itemize} 
The so-obtained coloring $\gamma$ is clearly a proper edge coloring of $(B(m,2r-1,k),\{v_1u_{1}^{1}\})$.

By \textbf{Cases 1} and \textbf{2}, the proof is complete.
\end{proof}

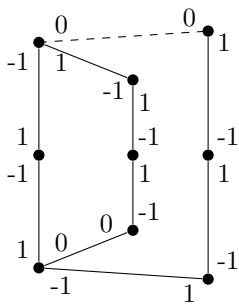
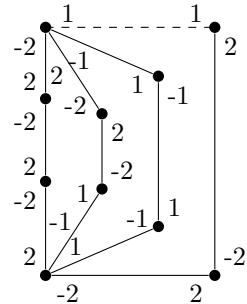
\begin{figure}[ht]
\hspace*{0.5cm}
\begin{subfigure}{0.4\textwidth}
\begin{tikzpicture}[scale=0.5]
\node[fill=black, circle, inner sep=1.5pt] (v1) at (9,0) {};
\node[fill=black, circle, inner sep=1.5pt] (v2) at (9,3) {};
\node[fill=black, circle, inner sep=1.5pt] (v3) at (9,6) {};
\node[fill=black, circle, inner sep=1.5pt] (v4) at (11.5,1) {};
\node[fill=black, circle, inner sep=1.5pt] (v5) at (11.5,3) {};
\node[fill=black, circle, inner sep=1.5pt] (v6) at (11.5,5) {};
\node[fill=black, circle, inner sep=1.5pt] (v7) at (13.5,-0.3) {};
\node[fill=black, circle, inner sep=1.5pt] (v8) at (13.5,3) {};
\node[fill=black, circle, inner sep=1.5pt] (v9) at (13.5,6.3) {};

\foreach \from/\to in {v1/v2,v1/v4,v2/v3,v1/v7,v3/v6,v4/v5,v5/v6,v7/v8,v8/v9} \draw (\from) -- (\to);
\draw [dashed] (9,6) -- (13.5,6.3);

\node [left] at (9,5.5) {-$1$};
\node [left] at (9,3.5) {$1$};
\node [left] at (9,2.5) {-$1$};
\node [left] at (9,0.5) {$1$};

\node [above] at (9.6,6) {$0$};
\node [above] at (13,6.2) {$0$};
\node [right] at (13.5,6) {$1$};
\node [right] at (13.5,3.5) {-$1$};
\node [right] at (13.5,2.5) {$1$};
\node [right] at (13.5,0.2) {-$1$};
\node [below] at (13,-0.2) {$1$};
\node [below] at (9.6,0) {-$1$};

\node [above] at (9.6,5) {$1$};
\node [below] at (11,5.2) {-$1$};
\node [right] at (11.4,4.4) {$1$};
\node [right] at (11.4,3.5) {-$1$};
\node [right] at (11.4,2.5) {$1$};
\node [right] at (11.4,1.5) {-$1$};
\node [above] at (10.8,0.7) {$0$};
\node [above] at (9.6,0.2) {$0$};

\end{tikzpicture}
\caption{Proper 3-edge coloring of $(B(6,2,3),\sigma_1)$.}\label{Fig-coloring of B(6,2,3) with sigma_1}
\end{subfigure}
\hfill
\hspace*{-1cm}
\begin{subfigure}{0.4\textwidth}
\begin{tikzpicture}[scale=0.5]
\node[fill=black, circle, inner sep=1.5pt] (v1) at (9,-0.3) {};
\node[fill=black, circle, inner sep=1.5pt] (v2) at (9,2.2) {};
\node[fill=black, circle, inner sep=1.5pt] (v3) at (9,4.4) {};
\node[fill=black, circle, inner sep=1.5pt] (v4) at (9,6.3) {};
\node[fill=black, circle, inner sep=1.5pt] (v5) at (10.5,2) {};
\node[fill=black, circle, inner sep=1.5pt] (v6) at (10.5,4) {};
\node[fill=black, circle, inner sep=1.5pt] (v7) at (12,1) {};
\node[fill=black, circle, inner sep=1.5pt] (v8) at (12,5) {};
\node[fill=black, circle, inner sep=1.5pt] (v9) at (13.5,-0.3) {};
\node[fill=black, circle, inner sep=1.5pt] (v10) at (13.5,6.3) {};

\foreach \from/\to in {v1/v2,v2/v3,v1/v5,v1/v7,v1/v9,v3/v4,v4/v6,v4/v8,v5/v6,v7/v8,v9/v10} \draw (\from) -- (\to);
\draw [dashed] (9,6.3) -- (13.5,6.3);


\node [left] at (9,5.8) {-$2$};
\node [left] at (9,4.85) {$2$};
\node [left] at (9,3.8) {-$2$};
\node [left] at (9,2.6) {$2$};
\node [left] at (9,1.7) {-$2$};
\node [left] at (9,0.2) {$2$};

\node [above] at (9.6,6.2) {$1$};
\node [above] at (13,6.2) {$1$};
\node [right] at (13.5,5.8) {$2$};
\node [right] at (13.5,0.2) {-$2$};
\node [below] at (13,-0.3) {$2$};
\node [below] at (9.6,-0.3) {-$2$};

\node [below] at (9.9,5.9) {-$1$};
\node [below] at (11.45,5.25) {$1$};
\node [right] at (12,4.5) {-$1$};
\node [right] at (12,1.5) {$1$};
\node [above] at (11.45,0.75) {-$1$};
\node [above] at (9.8,0) {$1$};

\node [below] at (9.3,5.5) {$2$};
\node [below] at (9.8,4.7) {-$2$};
\node [right] at (10.5,3.5) {$2$};
\node [right] at (10.5,2.5) {-$2$};
\node [above] at (10,1.4) {$1$};
\node [above] at (9.4,0.65) {-$1$};

\end{tikzpicture}
\caption{Proper 4-edge coloring of $(B(6,3,4),\sigma_1)$.}\label{Fig-coloring of B(6,3,4) with sigma_1}
\end{subfigure}
\caption{Proper edge colorings of some signed generalized book graphs with signature $\sigma_1$.}\label{Fig-edge colorings of B(6,2,3) and B(6,3,4) with sigma one}
\end{figure}

\begin{figure}[ht]
\hspace*{0.5cm}
\begin{subfigure}{0.4\textwidth}
\begin{tikzpicture}[scale=0.5]
\node[fill=black, circle, inner sep=1.5pt] (v1) at (9,-0.3) {};
\node[fill=black, circle, inner sep=1.5pt] (v2) at (9,3.3) {};
\node[fill=black, circle, inner sep=1.5pt] (v3) at (9,6.3) {};
\node[fill=black, circle, inner sep=1.5pt] (v4) at (10.5,2) {};
\node[fill=black, circle, inner sep=1.5pt] (v5) at (10.5,4) {};
\node[fill=black, circle, inner sep=1.5pt] (v6) at (12,1) {};
\node[fill=black, circle, inner sep=1.5pt] (v7) at (12,5) {};
\node[fill=black, circle, inner sep=1.5pt] (v8) at (13.5,-0.3) {};
\node[fill=black, circle, inner sep=1.5pt] (v9) at (13.5,6.3) {};
\node[fill=black, circle, inner sep=1.5pt] (v10) at (15,-2) {};
\node[fill=black, circle, inner sep=1.5pt] (v11) at (15,8) {};

\foreach \from/\to in {v1/v2,v2/v3,v1/v4,v1/v6,v1/v8,v1/v10,v3/v5,v3/v7,v4/v5,v6/v7,v8/v9,v10/v11} \draw (\from) -- (\to);
\draw [dashed] (9,6.3) -- (15,8);
\draw [dashed] (9,6.3) -- (13.5,6.3);


\node [left] at (9,5.8) {-$2$};
\node [left] at (9,3.8) {$2$};
\node [left] at (9,2.8) {-$2$};
\node [left] at (9,0.2) {$2$};

\node [above] at (9.6,6.5) {$0$};
\node [above] at (14.3,7.8) {$0$};
\node [right] at (15,7.4) {$1$};
\node [right] at (15,-1.3) {-$1$};
\node [below] at (14.3,-1.8) {$1$};
\node [below] at (9.6,-0.5) {-$1$};

\node [below] at (10.6,6.4) {$1$};
\node [below] at (13,6.3) {$1$};
\node [right] at (13.5,5.8) {-$1$};
\node [right] at (13.5,0.2) {$1$};
\node [above] at (13,-0.3) {-$1$};
\node [above] at (10.6,-0.4) {$1$};

\node [below] at (9.9,5.9) {-$1$};
\node [below] at (11.45,5.25) {$1$};
\node [right] at (12,4.5) {-$1$};
\node [right] at (12,1.5) {$1$};
\node [above] at (11.45,0.75) {$0$};
\node [above] at (9.8,0) {$0$};

\node [below] at (9.3,5.5) {$2$};
\node [below] at (9.8,4.7) {-$2$};
\node [right] at (10.5,3.5) {$2$};
\node [right] at (10.5,2.5) {-$2$};
\node [above] at (10,1.4) {$2$};
\node [above] at (9.4,0.65) {-$2$};

\end{tikzpicture}
\caption{Proper 5-edge coloring of $(B(5,4,3),\sigma_2)$.}\label{Fig-coloring of B(5,4,3) with sigma two}
\end{subfigure}
\hfill
\hspace*{-1cm}
\begin{subfigure}{0.4\textwidth}
\begin{tikzpicture}[scale=0.5]
\node[fill=black, circle, inner sep=1.5pt] (v1) at (9,-0.3) {};
\node[fill=black, circle, inner sep=1.5pt] (v2) at (9,3.3) {};
\node[fill=black, circle, inner sep=1.5pt] (v3) at (9,6.3) {};
\node[fill=black, circle, inner sep=1.5pt] (v4) at (10.5,2) {};
\node[fill=black, circle, inner sep=1.5pt] (v5) at (10.5,4) {};
\node[fill=black, circle, inner sep=1.5pt] (v6) at (12,1) {};
\node[fill=black, circle, inner sep=1.5pt] (v7) at (12,5) {};
\node[fill=black, circle, inner sep=1.5pt] (v8) at (13.5,-0.3) {};
\node[fill=black, circle, inner sep=1.5pt] (v9) at (13.5,6.3) {};
\node[fill=black, circle, inner sep=1.5pt] (v10) at (15,-2) {};
\node[fill=black, circle, inner sep=1.5pt] (v11) at (15,8) {};

\foreach \from/\to in {v1/v2,v2/v3,v1/v4,v1/v6,v1/v8,v1/v10,v3/v5,v4/v5,v6/v7,v8/v9,v10/v11} \draw (\from) -- (\to);
\draw [dashed] (9,6.3) -- (15,8);
\draw [dashed] (9,6.3) -- (13.5,6.3);
\draw [dashed] (9,6.3) -- (12,5);


\node [left] at (9,5.8) {-$2$};
\node [left] at (9,3.8) {$2$};
\node [left] at (9,2.8) {-$2$};
\node [left] at (9,0.2) {$2$};

\node [above] at (9.6,6.5) {$0$};
\node [above] at (14.3,7.8) {$0$};
\node [right] at (15,7.4) {$2$};
\node [right] at (15,-1.3) {-$2$};
\node [below] at (14.3,-1.8) {$2$};
\node [below] at (9.6,-0.5) {-$2$};

\node [below] at (10.6,6.4) {$1$};
\node [below] at (13,6.3) {$1$};
\node [right] at (13.5,5.8) {-$1$};
\node [right] at (13.5,0.2) {$1$};
\node [above] at (13,-0.3) {-$1$};
\node [above] at (10.6,-0.4) {$1$};

\node [below] at (9.9,5.9) {-$1$};
\node [below] at (11.45,5.25) {-$1$};
\node [right] at (12,4.5) {$1$};
\node [right] at (12,1.5) {-$1$};
\node [above] at (11.45,0.75) {$1$};
\node [above] at (9.8,0) {-$1$};

\node [below] at (9.3,5.5) {$2$};
\node [below] at (9.8,4.7) {-$2$};
\node [right] at (10.5,3.5) {$2$};
\node [right] at (10.5,2.5) {-$2$};
\node [above] at (10,1.4) {$0$};
\node [above] at (9.4,0.65) {$0$};

\end{tikzpicture}
\caption{Proper 5-edge coloring of $(B(5,4,3),\sigma_3)$.}\label{Fig-coloring of B(5,4,3) with sigma three}
\end{subfigure}
\caption{An illustration of Case 1 of Theorem~\ref{Theorem-chrom index of signed GBG with sigma l}. Proper colorings of $(B(5,4,3),\sigma_2)$ and $(B(5,4,3),\sigma_3)$.} \label{Figure - proper edge colorings for Case 1 of Thm 3.6}
\end{figure}
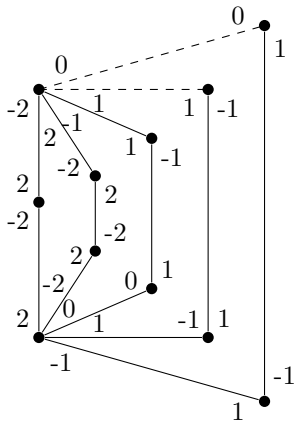
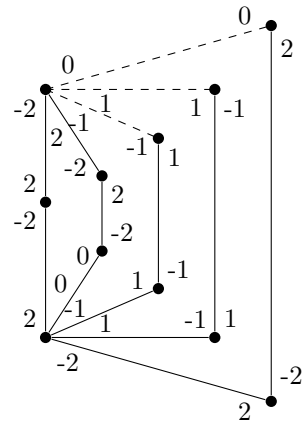

\begin{theorem}\label{Theorem-chrom index of signed GBG with sigma l}
Let $m \geq 3, n \geq 3, k \geq 2$, then $\chi'(B(m,n,k),\sigma_{l}) = n+1$, where $2 \leq l \leq n-1$.
\end{theorem}
\begin{proof}
It is clear that the set of negative edges in a signed generalized book graph with signature $\sigma_l$ is $\{v_1u_{1}^{1},\ldots , v_1u_{1}^{l}\}$, where $2 \leq l \leq n-1$. 

\noindent
\textbf{Case 1.} Suppose $n=2r$ for some integer $r \geq 1$.

\textbf{Subcase 1.1.} Let $l$ be even. Thus $l$ must lies between 2 and $2r-2$. In this case, we give a proper $(2r+1)$-edge coloring $\gamma$ of $(B(m,2r,k), \sigma_l)$ as follows (see Figure~\ref{Fig-coloring of B(5,4,3) with sigma two} for the case $m=5,n=4,k=3$).
\begin{itemize}
\item[1.] We sequentially color the incidences of the path $v_1 \ldots v_k$ using the pattern $((-r)(r)) \ldots ((-r)(r))$. 
\item[2.] For the edge $v_1u_{1}^{i}$, we set $$\gamma (v_1,v_1u_{1}^{i}) = \gamma (u_{1}^{i},u_{1}^{i}v_1) = \begin{cases} 0 ~~~~~~~~\text{if}~i=1,\\
\frac{i}{2}~~~~~~~~\text{if}~ i~\text{is even and}~2 \leq i \leq l,\\
- \frac{i-1}{2}~~~\text{if}~ i~\text{is odd and}~3 \leq i \leq l-1 ~(\text{this case occurs only if}~ l \geq 4).
\end{cases}$$
\item[3.] We sequentially color the incidences of $u_{1}^{1} \ldots u_{m-k}^{1} v_k$ using the pattern $((\frac{l}{2})(- \frac{l}{2})) \ldots ((\frac{l}{2})(- \frac{l}{2}))$.
\item[4.] For $ i = 2,\ldots , l$, we sequentially color the incidences of the path $u_{1}^{i} \cdots u_{m-k}^{i}v_k$ using the pattern $((-\frac{i}{2})( \frac{i}{2})) \ldots ((-\frac{i}{2})(\frac{i}{2}))$ and $((\frac{i-1}{2})(-\frac{i-1}{2})) \ldots ((\frac{i-1}{2})(-\frac{i-1}{2}))$  according as $i$ is even and odd, respectively.
\item[5.] We sequentially color the incidences of $v_1u_{1}^{l+1} \ldots u_{m-k}^{l+1}$ using the pattern $((-\frac{l}{2})(\frac{l}{2})) \ldots ((-\frac{l}{2})(\frac{l}{2}))$ and for the edge $u_{m-k}^{l+1}v_k$, we set $\gamma (u_{m-k}^{l+1},u_{m-k}^{l+1}v_k) = \gamma (v_k,v_ku_{m-k}^{l+1}) =0$.
\item[6.] For $ i = l+2,\ldots , 2r$, we sequentially color the incidences of the path $v_1u_{1}^{i} \cdots u_{m-k}^{i}v_k$ using the pattern $((\frac{i}{2})(- \frac{i}{2})) \ldots ((\frac{i}{2})(-\frac{i}{2}))$ and $((- \frac{i-1}{2})(\frac{i-1}{2})) \ldots ((-\frac{i-1}{2})(\frac{i-1}{2}))$  according as $i$ is even and odd, respectively.
\end{itemize} 
The so-obtained coloring $\gamma$ is clearly a proper edge coloring of $(B(m,2r,k),\sigma_l)$, where $l$ is even.

\textbf{Subcase 1.2.} Let $l$ be odd. Thus $l$ must lies between 3 and $2r-1$. In this case, we give a proper $(2r+1)$-edge coloring $\gamma$ of $(B(m,2r,k), \sigma_l)$ as follows (see Figure~\ref{Fig-coloring of B(5,4,3) with sigma three} for the case $m=5,n=4,k=3$).
\begin{itemize}
\item[1.] We sequentially color the incidences of the path $v_1 \ldots v_k$ using the pattern $((-r)(r)) \ldots ((-r)(r))$. 
\item[2.] For the edge $v_1u_{1}^{i}$, we set $$\gamma (v_1,v_1u_{1}^{i}) = \gamma (u_{1}^{i},u_{1}^{i}v_1) = \begin{cases} 0 ~~~~~~~~\text{if}~i=1,\\
\frac{i}{2}~~~~~~~~\text{if}~ i~\text{is even and}~2 \leq i \leq l-1,\\
- \frac{i-1}{2}~~~\text{if}~ i~\text{is odd and}~3 \leq i \leq l.
\end{cases}$$
\item[3.] We sequentially color the incidences of $u_{1}^{1} \ldots u_{m-k}^{1} v_k$ using the pattern $((r)(-r)) \ldots ((r)(-r))$.
\item[4.] For $ i = 2,\ldots , l$, we sequentially color the incidences of the path $u_{1}^{i} \cdots u_{m-k}^{i}v_k$ using the pattern $((-\frac{i}{2})( \frac{i}{2})) \ldots ((-\frac{i}{2})(\frac{i}{2}))$ and $((\frac{i-1}{2})(-\frac{i-1}{2})) \ldots ((\frac{i-1}{2})(-\frac{i-1}{2}))$  according as $i$ is even and odd, respectively.
\item[5.] If there is an integer $i$ such that $l+1 \leq i \leq  2r-1$, then we sequentially color the incidences of the path $v_1u_{1}^{i} \cdots u_{m-k}^{i}v_k$ using the pattern $((\frac{i}{2})(- \frac{i}{2})) \ldots ((\frac{i}{2})(-\frac{i}{2}))$ and $((- \frac{i-1}{2})(\frac{i-1}{2})) \ldots ((-\frac{i-1}{2})(\frac{i-1}{2}))$  according as $i$ is even and odd, respectively.
\item[6.] We sequentially color the incidences of $v_1u_{1}^{2r} \ldots u_{m-k}^{2r}$ using the pattern $((r)(-r)) \ldots ((r)(-r))$ and for the edge $u_{m-k}^{2r}v_k$, we set $\gamma (u_{m-k}^{2r},u_{m-k}^{2r}v_k) = \gamma (v_k,v_ku_{m-k}^{2r}) =0$.
\end{itemize}
The so-obtained coloring $\gamma$ is clearly a proper edge coloring of $(B(m,2r,k),\sigma_l)$, where $l$ is odd.

\noindent
\textbf{Case 2.} Suppose $n=2r-1$ for some integer $r \geq 2$.

\textbf{Subcase 2.1.} Let $l$ be even. Thus $l$ must lies between 2 and $2r-2$. In this case, we give a proper $(2r)$-edge coloring $\gamma$ of $(B(m,2r-1,k), \sigma_l)$ as follows (see Figure~\ref{Fig-coloring of B(5,3,3) with sigma two} for the case $m=5,n=3,k=3$).
\begin{itemize}
\item[1.] We sequentially color the incidences of the path $v_1 \ldots v_k$ using the pattern $((-r)(r)) \ldots ((-r)(r))$. 
\item[2.] For the edge $v_1u_{1}^{i}$, we set $$\gamma (v_1,v_1u_{1}^{i}) = \gamma (u_{1}^{i},u_{1}^{i}v_1) = \begin{cases} 
\frac{i+1}{2}~~~~~\text{if}~ i~\text{is odd and}~1 \leq i \leq l-1,\\
 -\frac{i}{2}~~~~~\text{if}~ i~\text{is even and}~2 \leq i \leq l.
\end{cases}$$
\item[3.] For $ i = 1,\ldots , l$, we sequentially color the incidences of the path $u_{1}^{i} \cdots u_{m-k}^{i}v_k$ using the pattern $((-\frac{i+1}{2})(\frac{i+1}{2})) \ldots ((-\frac{i+1}{2})(\frac{i+1}{2}))$ and $((\frac{i}{2})(- \frac{i}{2})) \ldots ((\frac{i}{2})(-\frac{i}{2}))$ according as $i$ is odd and even, respectively.
\item[4.] For $ i = l+1,\ldots , 2r-1$, we sequentially color the incidences of the path $v_1u_{1}^{i} \cdots u_{m-k}^{i}v_k$ using the pattern $((\frac{i+1}{2})(-\frac{i+1}{2})) \ldots ((\frac{i+1}{2})(-\frac{i+1}{2}))$ and $((-\frac{i}{2})( \frac{i}{2})) \ldots ((-\frac{i}{2})(\frac{i}{2}))$ according as $i$ is odd and even, respectively.
\end{itemize} 

The so-obtained coloring $\gamma$ is clearly a proper edge coloring of $(B(m,2r-1,k),\sigma_l)$, where $l$ is even.

\textbf{Subcase 2.2.} Let $l$ be odd. Thus $l$ must lies between 3 and $2r-3$. In this case, we give a proper $(2r)$-edge coloring $\gamma$ of $(B(m,2r-1,k), \sigma_l)$ as follows (see Figure~\ref{Fig-coloring of B(5,5,3) with sigma three} for the case $m=5,n=5,k=3$).
\begin{itemize}
\item[1.] We sequentially color the incidences of the path $v_1 \ldots v_k$ using the pattern $((-r)(r)) \ldots ((-r)(r))$. 
\item[2.] For the edge $v_1u_{1}^{i}$, we set $$\gamma (v_1,v_1u_{1}^{i}) = \gamma (u_{1}^{i},u_{1}^{i}v_1) = \begin{cases} 
\frac{i+1}{2}~~~~~\text{if}~ i~\text{is odd and}~1 \leq i \leq l,\\
 -\frac{i}{2}~~~~~\text{if}~ i~\text{is even and}~2 \leq i \leq l-1.
\end{cases}$$
\item[3.] We color the incidences of $u_{1}^{1} \ldots u_{m-k}^{1} v_k$ using the pattern $((\frac{l+1}{2})(-\frac{l+1}{2})) \ldots ((\frac{l+1}{2})(-\frac{l+1}{2}))$.
\item[4.] For $ i = 2,\ldots , l-1$, we sequentially color the incidences of the path $u_{1}^{i} \cdots u_{m-k}^{i}v_k$ using the pattern $((\frac{i}{2})(- \frac{i}{2})) \ldots ((\frac{i}{2})(-\frac{i}{2}))$ and $((-\frac{i+1}{2})(\frac{i+1}{2})) \ldots ((-\frac{i+1}{2})(\frac{i+1}{2}))$ according as $i$ is even and odd, respectively.
\item[5.] We sequentially color the incidences of $u_{1}^{l} \ldots u_{m-k}^{l} v_k$ using the pattern $((-1)(1)) \ldots ((-1)(1))$.
\item[6.] For $ i = l+1,\ldots , 2r-1$, we sequentially color the incidences of the path $v_1u_{1}^{i} \cdots u_{m-k}^{i}v_k$ using the pattern $((-\frac{i}{2})( \frac{i}{2})) \ldots ((-\frac{i}{2})(\frac{i}{2}))$ and $((\frac{i+1}{2})(-\frac{i+1}{2})) \ldots ((\frac{i+1}{2})(-\frac{i+1}{2}))$ according as $i$ is even and odd, respectively.
\end{itemize}

The so-obtained coloring $\gamma$ is clearly a proper edge coloring of $(B(m,2r-1,k),\sigma_l)$, where $l$ is odd.

Hence the proof follows from Cases 1 and 2.
\end{proof}

\begin{figure}[ht]
\begin{subfigure}{0.4\textwidth}
\begin{tikzpicture}[scale=0.6]
\node[fill=black, circle, inner sep=1.5pt] (v1) at (9,-0.3) {};
\node[fill=black, circle, inner sep=1.5pt] (v2) at (9,3.3) {};
\node[fill=black, circle, inner sep=1.5pt] (v3) at (9,6.3) {};
\node[fill=black, circle, inner sep=1.5pt] (v4) at (10.5,2) {};
\node[fill=black, circle, inner sep=1.5pt] (v5) at (10.5,4) {};
\node[fill=black, circle, inner sep=1.5pt] (v6) at (12,1) {};
\node[fill=black, circle, inner sep=1.5pt] (v7) at (12,5) {};
\node[fill=black, circle, inner sep=1.5pt] (v8) at (13.5,-0.3) {};
\node[fill=black, circle, inner sep=1.5pt] (v9) at (13.5,6.3) {};

\foreach \from/\to in {v1/v2,v2/v3,v1/v4,v1/v6,v1/v8,v3/v5,v4/v5,v6/v7,v8/v9} \draw (\from) -- (\to);
\draw [dashed] (9,6.3) -- (12,5);
\draw [dashed] (9,6.3) -- (13.5,6.3);


\node [left] at (9,5.8) {-$2$};
\node [left] at (9,3.8) {$2$};
\node [left] at (9,2.8) {-$2$};
\node [left] at (9,0.2) {$2$};

\node [above] at (10,6.3) {$1$};
\node [above] at (13,6.3) {$1$};
\node [right] at (13.5,5.8) {-$1$};
\node [right] at (13.5,0.2) {$1$};
\node [below] at (13,-0.3) {-$1$};
\node [below] at (10,-0.3) {$1$};

\node [below] at (9.9,5.9) {-$1$};
\node [below] at (11.45,5.25) {-$1$};
\node [right] at (12,4.5) {$1$};
\node [right] at (12,1.5) {-$1$};
\node [above] at (11.45,0.75) {$1$};
\node [above] at (9.8,0) {-$1$};

\node [below] at (9.3,5.5) {$2$};
\node [below] at (9.8,4.7) {-$2$};
\node [right] at (10.5,3.5) {$2$};
\node [right] at (10.5,2.5) {-$2$};
\node [above] at (10,1.4) {$2$};
\node [above] at (9.4,0.65) {-$2$};

\end{tikzpicture}
\caption{Proper 4-edge coloring of $(B(5,3,3),\sigma_2)$.}\label{Fig-coloring of B(5,3,3) with sigma two}
\end{subfigure}
\hfill
\hspace*{-1cm}
\begin{subfigure}{0.4\textwidth}
\begin{tikzpicture}[scale=0.6]
\node[fill=black, circle, inner sep=1.5pt] (v1) at (9,-0.3) {};
\node[fill=black, circle, inner sep=1.5pt] (v2) at (9,3.3) {};
\node[fill=black, circle, inner sep=1.5pt] (v3) at (9,6.3) {};
\node[fill=black, circle, inner sep=1.5pt] (v4) at (10.5,2) {};
\node[fill=black, circle, inner sep=1.5pt] (v5) at (10.5,4) {};
\node[fill=black, circle, inner sep=1.5pt] (v6) at (11.5,1) {};
\node[fill=black, circle, inner sep=1.5pt] (v7) at (11.5,5) {};
\node[fill=black, circle, inner sep=1.5pt] (v8) at (12.5,0) {};
\node[fill=black, circle, inner sep=1.5pt] (v9) at (12.5,6) {};
\node[fill=black, circle, inner sep=1.5pt] (v10) at (13.5,-1.3) {};
\node[fill=black, circle, inner sep=1.5pt] (v11) at (13.5,7.3) {};
\node[fill=black, circle, inner sep=1.5pt] (v12) at (14.5,-2.5) {};
\node[fill=black, circle, inner sep=1.5pt] (v13) at (14.5,8.5) {};

\foreach \from/\to in {v1/v2,v2/v3,v1/v4,v1/v6,v1/v8,v1/v10,v1/v12,v3/v5,v3/v7,v4/v5,v6/v7,v8/v9,v10/v11,v12/v13} \draw (\from) -- (\to);
\draw [dashed] (9,6.3) -- (13.5,7.3);
\draw [dashed] (9,6.3) -- (12.5,6);
\draw [dashed] (9,6.3) -- (14.5,8.5);


\node [left] at (9,5.8) {-$3$};
\node [left] at (9,3.8) {$3$};
\node [left] at (9,2.8) {-$3$};
\node [left] at (9,0.2) {$3$};

\node [above] at (9.8,6.6) {$1$};
\node [above] at (13.8,8.2) {$1$};
\node [right] at (14.5,7.8) {$2$};
\node [right] at (14.5,-1.8) {-$2$};
\node [below] at (14,-2.2) {$2$};
\node [below] at (9.8,-0.8) {-$2$};

\node [below] at (10.9,6.8) {-$1$};
\node [below] at (12.8,7.2) {-$1$};
\node [right] at (13.5,6.8) {$1$};
\node [right] at (13.5,-0.8) {-$1$};
\node [above] at (12.8,-1.1) {$1$};
\node [above] at (10.9,-0.8) {-$1$};

\node [below] at (10.6,6.2) {$2$};
\node [below] at (12,6) {$2$};
\node [right] at (12.5,5.3) {-$1$};
\node [right] at (12.5,0.6) {$1$};
\node [above] at (12,0) {-$1$};
\node [above] at (10.6,-0.2) {$1$};

\node [below] at (9.9,5.9) {-$2$};
\node [below] at (11,5.25) {$2$};
\node [right] at (11.5,4.5) {-$2$};
\node [right] at (11.5,1.5) {$2$};
\node [above] at (11,0.75) {-$2$};
\node [above] at (9.8,0) {$2$};

\node [below] at (9.3,5.5) {$3$};
\node [below] at (9.8,4.7) {-$3$};
\node [right] at (10.5,3.5) {$3$};
\node [right] at (10.5,2.5) {-$3$};
\node [above] at (10,1.4) {$3$};
\node [above] at (9.4,0.65) {-$3$};
\end{tikzpicture}
\caption{Proper 5-edge coloring of $(B(5,5,3),\sigma_3)$.}\label{Fig-coloring of B(5,5,3) with sigma three}
\end{subfigure}
\caption{An illustration of Case 2 of Theorem~\ref{Theorem-chrom index of signed GBG with sigma l}. Proper colorings of $(B(5,3,3),\sigma_2)$ and $(B(5,5,3),\sigma_3)$.} \label{Figure - proper edge colorings for Case 2 of Thm 3.6}
\end{figure}
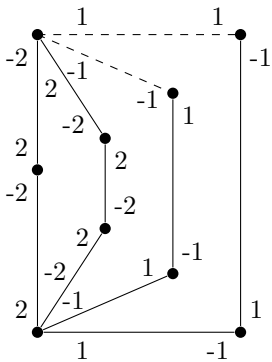
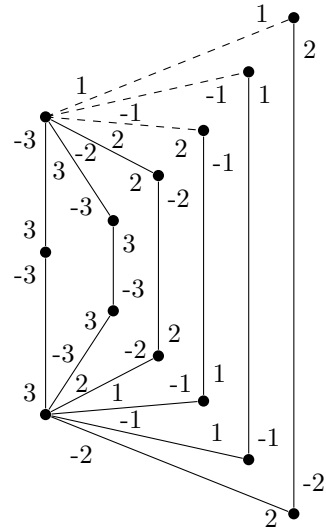

From Theorems~\ref{Theorem-chrom index of unsigned GBG} - \ref{Theorem-chrom index of signed GBG with sigma l}, it follows that each signed generalized book graph is class 1.

\subsection{Signed complete graphs}

In this subsection, we study edge coloring of signed complete graphs of order up to six.

It is well-known that there are two, three,  seven and sixteen signed complete graphs on three, four, five and six vertices, respectively. Their representatives, taken from \cite{Sehrawat2025}, are shown in Figures~\ref{Figure-proper edge colorings of signed K_n up to 5}, \ref{Figure-1-SCGs on 6 vertices} and \ref{Figure-2-SCGs on 6 vertices}. We determine the chromatic index of each of these representatives.

\begin{figure}[ht]
\begin{subfigure}{0.2\textwidth}
\centering
\begin{tikzpicture}[scale=0.35]
\draw[fill=black] (13,5) circle (6pt);
\node [right] at (13.5,4.3) {$1$};
\node [right] at (15.2,1) {$-1$};
\draw[fill=black] (16,0) circle (6pt);
\node [below] at (15,0) {$1$};
\node [below] at (11,0) {$-1$};
\draw[fill=black] (10,0) circle (6pt);
\node [left] at (10.6,1) {$1$};
\node [left] at (12.4,4.3) {$-1$};
\node [below] at (13,-1.3) {$(K_{3},\sigma_{1})$};

\draw (10,0) -- (13,5);
\draw (10,0) -- (16,0);
\draw (13,5) -- (16,0);
\end{tikzpicture}
\end{subfigure}
\hfill
\begin{subfigure}{0.2\textwidth}
\centering
\begin{tikzpicture}[scale=0.35]
\draw[fill=black] (13,5) circle (6pt);
\node [right] at (13.5,4.3) {$0$};
\node [right] at (15.4,1) {$0$};
\draw[fill=black] (16,0) circle (6pt);
\node [below] at (15,0) {$1$};
\node [below] at (11,0) {$-1$};
\draw[fill=black] (10,0) circle (6pt);
\node [left] at (10.6,1) {$1$};
\node [left] at (12.4,4.3) {$-1$};
\node [below] at (13,-1.3) {{$(K_{3},\sigma_{2})$}};

\draw (10,0) -- (13,5);
\draw (10,0) -- (16,0);
\draw [dashed] (13,5) -- (16,0);
\end{tikzpicture}
\end{subfigure}
\hfill
\begin{subfigure}{0.2\textwidth}
\centering
\begin{tikzpicture}[scale=0.35]
\draw[fill=black] (10,5) circle (6pt);
\draw[fill=black] (16,5) circle (6pt);
\draw[fill=black] (16,0) circle (6pt);
\draw[fill=black] (10,0) circle (6pt);
\node [below] at (13,-1.3) {$(K_{4},\sigma_1)$};

\node [above] at (11,5) {$1$};
\node [above] at (15,5) {$-1$};
\node [right] at (16,4) {$1$};
\node [right] at (16,1) {$-1$};
\node [below] at (15,0) {$1$};
\node [below] at (11,0) {$-1$};
\node [left] at (10,1) {$1$};
\node [left] at (10,4) {$-1$};
\node [below] at (15.2,4.2) {$0$};
\node [above] at (10.8,0.7) {$0$};
\node [below] at (10.8,4.3) {$0$};
\node [above] at (15.2,0.6) {$0$};

\draw (10,5) -- (16,5);
\draw (10,0) -- (16,5);
\draw (10,5) -- (16,0);
\draw (10,0) -- (16,0);
\draw (10,5) -- (10,0);
\draw (16,5) -- (16,0);
\end{tikzpicture}
\end{subfigure}
\hfill
\begin{subfigure}{0.2\textwidth}
\centering
\begin{tikzpicture}[scale=0.35]
\draw[fill=black] (10,5) circle (6pt);
\draw[fill=black] (16,5) circle (6pt);
\draw[fill=black] (16,0) circle (6pt);
\draw[fill=black] (10,0) circle (6pt);
\node [below] at (13,-1.3) {$(K_{4},\sigma_2)$};

\node [above] at (11,5) {$1$};
\node [above] at (15,5) {$-1$};
\node [right] at (16,4) {$0$};
\node [right] at (16,1) {$0$};
\node [below] at (15,0) {$-1$};
\node [below] at (11,0) {$1$};
\node [left] at (10,1) {$0$};
\node [left] at (10,4) {$0$};
\node [below] at (15.2,4.2) {$1$};
\node [above] at (10.8,0.7) {$-1$};
\node [below] at (10.8,4.3) {$-1$};
\node [above] at (15.2,0.6) {$1$};

\draw (10,5) -- (16,5);
\draw (10,0) -- (16,5);
\draw (10,5) -- (16,0);
\draw (10,0) -- (16,0);
\draw (10,5) -- (10,0);
\draw [dashed] (16,5) -- (16,0);

\end{tikzpicture}
\end{subfigure}
\hfill
\vspace{0.3in}
\begin{subfigure}{0.2\textwidth}
\centering
\begin{tikzpicture}[scale=0.35]
\draw[fill=black] (10,5) circle (6pt);
\draw[fill=black] (16,5) circle (6pt);
\draw[fill=black] (16,0) circle (6pt);
\draw[fill=black] (10,0) circle (6pt);
\node [below] at (13,-1.3) {$(K_{4},\sigma_3)$};

\node [above] at (11,5) {$1$};
\node [above] at (15,5) {$-1$};
\node [right] at (16,4) {$0$};
\node [right] at (16,1) {$0$};
\node [below] at (15,0) {$-1$};
\node [below] at (11,0) {$1$};
\node [left] at (10,1) {$0$};
\node [left] at (10,4) {$0$};
\node [below] at (15.2,4.2) {$1$};
\node [above] at (10.8,0.7) {$-1$};
\node [below] at (10.8,4.3) {$-1$};
\node [above] at (15.2,0.6) {$1$};

\draw (10,5) -- (16,5);
\draw (10,0) -- (16,5);
\draw (10,5) -- (16,0);
\draw (10,0) -- (16,0);
\draw [dashed] (10,5) -- (10,0);
\draw [dashed] (16,5) -- (16,0);
\end{tikzpicture}
\end{subfigure}
\hfill
\begin{subfigure}{0.2\textwidth}
\centering
\begin{tikzpicture}[scale=0.45]
\draw[fill=black] (10,0) circle (6pt);
\draw[fill=black] (14,0) circle (6pt);
\draw[fill=black] (9,3) circle (6pt);
\draw[fill=black] (15,3) circle (6pt);
\draw[fill=black] (12,5.5) circle (6pt);
\node [below] at (12,-1.3) {$(K_{5},\sigma_1)$};

\node [right] at (12.5,5.2) {$1$};
\node [right] at (14,4) {$-1$};
\node [right] at (14.6,2) {$1$};
\node [right] at (14.1,0.6) {$-1$};
\node [below] at (13,0) {$1$};
\node [below] at (11,0) {$-1$};
\node [left] at (9.8,0.8) {$1$};
\node [left] at (9.4,2.1) {$-1$};
\node [left] at (9.7,3.8) {$1$};
\node [left] at (11.2,5) {$-1$};

\node [right] at (12.3,4.5) {$2$};
\node [above] at (13.5,2.9) {-$2$};
\node [below] at (14.3,2.8) {$2$};
\node [right] at (13.2,1.4) {-$2$};
\node [above] at (12.5,-0.2) {$2$};
\node [above] at (11,0.5) {-$2$};
\node [right] at (9.6,1.2) {$2$};
\node [right] at (9.7,2.6) {-$2$};
\node [above] at (10.3,2.9) {$2$};
\node [right] at (10.5,4.4) {-$2$};

\draw (10,0) -- (9,3);
\draw (10,0) -- (12,5.5);
\draw (10,0) -- (14,0);
\draw (10,0) -- (15,3);
\draw (9,3) -- (12,5.5);
\draw (9,3) -- (15,3);
\draw (9,3) -- (14,0);
\draw (12,5.5) -- (15,3);
\draw (12,5.5) -- (14,0);
\draw (14,0) -- (15,3);
\end{tikzpicture}
\end{subfigure}
\hfill
\begin{subfigure}{0.2\textwidth}
\centering
\begin{tikzpicture}[scale=0.45]
\draw[fill=black] (10,0) circle (6pt);
\draw[fill=black] (14,0) circle (6pt);
\draw[fill=black] (9,3) circle (6pt);
\draw[fill=black] (15,3) circle (6pt);
\draw[fill=black] (12,5.5) circle (6pt);
\node [below] at (12,-1.3) {$(K_{5},\sigma_2)$};

\node [right] at (12.5,5.2) {$0$};
\node [right] at (14,4) {$0$};
\node [right] at (14.6,2) {$1$};
\node [right] at (14.1,0.6) {$-1$};
\node [below] at (13,0) {$1$};
\node [below] at (11,0) {$-1$};
\node [left] at (9.8,0.8) {$1$};
\node [left] at (9.4,2.1) {$-1$};
\node [left] at (9.7,3.8) {$1$};
\node [left] at (11.2,5) {$-1$};

\node [right] at (12.3,4.5) {$2$};
\node [above] at (13.5,2.9) {-$2$};
\node [below] at (14.3,2.8) {$2$};
\node [right] at (13.2,1.4) {-$2$};
\node [above] at (12.5,-0.2) {$2$};
\node [above] at (11,0.5) {-$2$};
\node [right] at (9.6,1.2) {$2$};
\node [right] at (9.7,2.6) {-$2$};
\node [above] at (10.3,2.9) {$2$};
\node [right] at (10.5,4.4) {-$2$};

\draw (10,0) -- (9,3);
\draw (10,0) -- (12,5.5);
\draw (10,0) -- (14,0);
\draw (10,0) -- (15,3);
\draw (9,3) -- (12,5.5);
\draw (9,3) -- (15,3);
\draw (9,3) -- (14,0);
\draw [dashed] (12,5.5) -- (15,3);
\draw (12,5.5) -- (14,0);
\draw (14,0) -- (15,3);
\end{tikzpicture}
\end{subfigure}
\hfill
\begin{subfigure}{0.2\textwidth}
\centering
\begin{tikzpicture}[scale=0.45]
\draw[fill=black] (10,0) circle (6pt);
\draw[fill=black] (14,0) circle (6pt);
\draw[fill=black] (9,3) circle (6pt);
\draw[fill=black] (15,3) circle (6pt);
\draw[fill=black] (12,5.5) circle (6pt);
\node [below] at (12,-1.3) {$(K_{5},\sigma_3)$};

\node [right] at (12.5,5.2) {$1$};
\node [right] at (14,4) {$1$};
\node [right] at (14.6,2) {$-1$};
\node [right] at (14.1,0.6) {$1$};
\node [below] at (13,0) {$-1$};
\node [below] at (11,0) {$-1$};
\node [left] at (9.8,0.8) {$1$};
\node [left] at (9.4,2.1) {$-1$};
\node [left] at (9.7,3.8) {$1$};
\node [left] at (11.2,5) {$-1$};

\node [right] at (12.3,4.5) {$2$};
\node [above] at (13.5,2.9) {-$2$};
\node [below] at (14.3,2.8) {$2$};
\node [right] at (13.2,1.4) {-$2$};
\node [above] at (12.5,-0.2) {$2$};
\node [above] at (11,0.5) {-$2$};
\node [right] at (9.6,1.2) {$2$};
\node [right] at (9.7,2.6) {-$2$};
\node [above] at (10.3,2.9) {$2$};
\node [right] at (10.5,4.4) {-$2$};

\draw (10,0) -- (9,3);
\draw (10,0) -- (12,5.5);
\draw [dashed] (10,0) -- (14,0);
\draw (10,0) -- (15,3);
\draw (9,3) -- (12,5.5);
\draw (9,3) -- (15,3);
\draw (9,3) -- (14,0);
\draw [dashed] (12,5.5) -- (15,3);
\draw (12,5.5) -- (14,0);
\draw (14,0) -- (15,3);
\end{tikzpicture}
\end{subfigure}
\hfill
\vspace{0.3in}
\begin{subfigure}{0.2\textwidth}
\centering
\begin{tikzpicture}[scale=0.45]
\draw[fill=black] (10,0) circle (6pt);
\draw[fill=black] (14,0) circle (6pt);
\draw[fill=black] (9,3) circle (6pt);
\draw[fill=black] (15,3) circle (6pt);
\draw[fill=black] (12,5.5) circle (6pt);
\node [below] at (12,-1.3) {$(K_{5},\sigma_4)$};

\node [right] at (12.5,5.2) {$1$};
\node [right] at (14,4) {$1$};
\node [right] at (14.6,2) {$-1$};
\node [right] at (14.1,0.6) {$-1$};
\node [below] at (13,0) {$1$};
\node [below] at (11,0) {$-1$};
\node [left] at (9.8,0.8) {$1$};
\node [left] at (9.4,2.1) {$-1$};
\node [left] at (9.7,3.8) {$1$};
\node [left] at (11.2,5) {$-1$};

\node [right] at (12.3,4.5) {$2$};
\node [above] at (13.5,2.9) {-$2$};
\node [below] at (14.3,2.8) {$2$};
\node [right] at (13.2,1.4) {-$2$};
\node [above] at (12.5,-0.2) {$2$};
\node [above] at (11,0.5) {-$2$};
\node [right] at (9.6,1.2) {$2$};
\node [right] at (9.7,2.6) {-$2$};
\node [above] at (10.3,2.9) {$2$};
\node [right] at (10.5,4.4) {-$2$};

\draw (10,0) -- (9,3);
\draw (10,0) -- (12,5.5);
\draw (10,0) -- (14,0);
\draw (10,0) -- (15,3);
\draw (9,3) -- (12,5.5);
\draw (9,3) -- (15,3);
\draw (9,3) -- (14,0);
\draw [dashed] (12,5.5) -- (15,3);
\draw (12,5.5) -- (14,0);
\draw [dashed] (14,0) -- (15,3);
\end{tikzpicture}
\end{subfigure}
\hfill
\begin{subfigure}{0.2\textwidth}
\centering
\begin{tikzpicture}[scale=0.45]
\draw[fill=black] (10,0) circle (6pt);
\draw[fill=black] (14,0) circle (6pt);
\draw[fill=black] (9,3) circle (6pt);
\draw[fill=black] (15,3) circle (6pt);
\draw[fill=black] (12,5.5) circle (6pt);
\node [below] at (12,-1.3) {$(K_{5},\sigma_5)$};

\node [right] at (12.5,5.2) {$1$};
\node [right] at (14,4) {$1$};
\node [right] at (14.6,2) {$-1$};
\node [right] at (14.1,0.6) {$-1$};
\node [below] at (13,0) {$1$};
\node [below] at (11,0) {$-1$};
\node [left] at (9.8,0.8) {$0$};
\node [left] at (9.4,2.1) {$0$};
\node [left] at (9.7,3.8) {$1$};
\node [left] at (11.2,5) {$-1$};

\node [right] at (12.3,4.5) {$2$};
\node [above] at (13.5,2.9) {-$2$};
\node [below] at (14.3,2.8) {$2$};
\node [right] at (13.2,1.4) {-$2$};
\node [above] at (12.5,-0.2) {$2$};
\node [above] at (11,0.5) {-$2$};
\node [right] at (9.6,1.2) {$2$};
\node [right] at (9.7,2.6) {-$2$};
\node [above] at (10.3,2.9) {$2$};
\node [right] at (10.5,4.4) {-$2$};

\draw [dashed] (10,0) -- (9,3);
\draw (10,0) -- (12,5.5);
\draw (10,0) -- (14,0);
\draw (10,0) -- (15,3);
\draw (9,3) -- (12,5.5);
\draw (9,3) -- (15,3);
\draw (9,3) -- (14,0);
\draw [dashed] (12,5.5) -- (15,3);
\draw (12,5.5) -- (14,0);
\draw [dashed] (14,0) -- (15,3);
\end{tikzpicture}
\end{subfigure}
\hfill
\begin{subfigure}{0.2\textwidth}
\centering
\begin{tikzpicture}[scale=0.45]
\draw[fill=black] (10,0) circle (6pt);
\draw[fill=black] (14,0) circle (6pt);
\draw[fill=black] (9,3) circle (6pt);
\draw[fill=black] (15,3) circle (6pt);
\draw[fill=black] (12,5.5) circle (6pt);
\node [below] at (12,-1.3) {$(K_{5},\sigma_6)$};

\node [right] at (12.5,5.2) {$1$};
\node [right] at (14,4) {$1$};
\node [right] at (14.6,2) {$-1$};
\node [right] at (14.1,0.6) {$-1$};
\node [below] at (13,0) {$0$};
\node [below] at (11,0) {$0$};
\node [left] at (9.8,0.8) {$1$};
\node [left] at (9.4,2.1) {$-1$};
\node [left] at (9.7,3.8) {$1$};
\node [left] at (11.2,5) {$-1$};

\node [right] at (12.3,4.5) {$2$};
\node [above] at (13.5,2.9) {-$2$};
\node [below] at (14.3,2.8) {$2$};
\node [right] at (13.2,1.4) {-$2$};
\node [above] at (12.5,-0.2) {$2$};
\node [above] at (11,0.5) {-$2$};
\node [right] at (9.6,1.2) {$2$};
\node [right] at (9.7,2.6) {-$2$};
\node [above] at (10.3,2.9) {$2$};
\node [right] at (10.5,4.4) {-$2$};

\draw (10,0) -- (9,3);
\draw (10,0) -- (12,5.5);
\draw [dashed] (10,0) -- (14,0);
\draw (10,0) -- (15,3);
\draw (9,3) -- (12,5.5);
\draw (9,3) -- (15,3);
\draw (9,3) -- (14,0);
\draw [dashed] (12,5.5) -- (15,3);
\draw (12,5.5) -- (14,0);
\draw [dashed] (14,0) -- (15,3);
\end{tikzpicture}
\end{subfigure}
\hfill
\begin{subfigure}{0.2\textwidth}
\centering
\begin{tikzpicture}[scale=0.45]
\draw[fill=black] (10,0) circle (6pt);
\draw[fill=black] (14,0) circle (6pt);
\draw[fill=black] (9,3) circle (6pt);
\draw[fill=black] (15,3) circle (6pt);
\draw[fill=black] (12,5.5) circle (6pt);
\node [below] at (12,-1.3) {$(K_{5},\sigma_7)$};

\node [right] at (12.5,5.2) {$1$};
\node [right] at (14,4) {$1$};
\node [right] at (14.6,2) {$-1$};
\node [right] at (14.1,0.6) {$-1$};
\node [below] at (13,0) {$1$};
\node [below] at (11,0) {$-1$};
\node [left] at (9.8,0.8) {$0$};
\node [left] at (9.4,2.1) {$0$};
\node [left] at (9.7,3.8) {$1$};
\node [left] at (11.2,5) {$-1$};

\node [right] at (12.3,4.5) {$0$};
\node [above] at (13.5,2.9) {-$2$};
\node [below] at (14.3,2.8) {$2$};
\node [right] at (13.4,1.4) {$0$};
\node [above] at (12.5,-0.2) {$2$};
\node [above] at (11,0.5) {-$2$};
\node [right] at (9.6,1.2) {$2$};
\node [right] at (9.7,2.6) {-$2$};
\node [above] at (10.3,2.9) {$2$};
\node [right] at (10.5,4.4) {-$2$};

\draw [dashed] (10,0) -- (9,3);
\draw (10,0) -- (12,5.5);
\draw (10,0) -- (14,0);
\draw (10,0) -- (15,3);
\draw (9,3) -- (12,5.5);
\draw (9,3) -- (15,3);
\draw (9,3) -- (14,0);
\draw [dashed] (12,5.5) -- (15,3);
\draw [dashed] (12,5.5) -- (14,0);
\draw [dashed] (14,0) -- (15,3);
\end{tikzpicture}
\end{subfigure}
\caption{Switching non-isomorphic signed complete graphs of order up to five and their proper edge colorings.}
\label{Figure-proper edge colorings of signed K_n up to 5}
\end{figure}
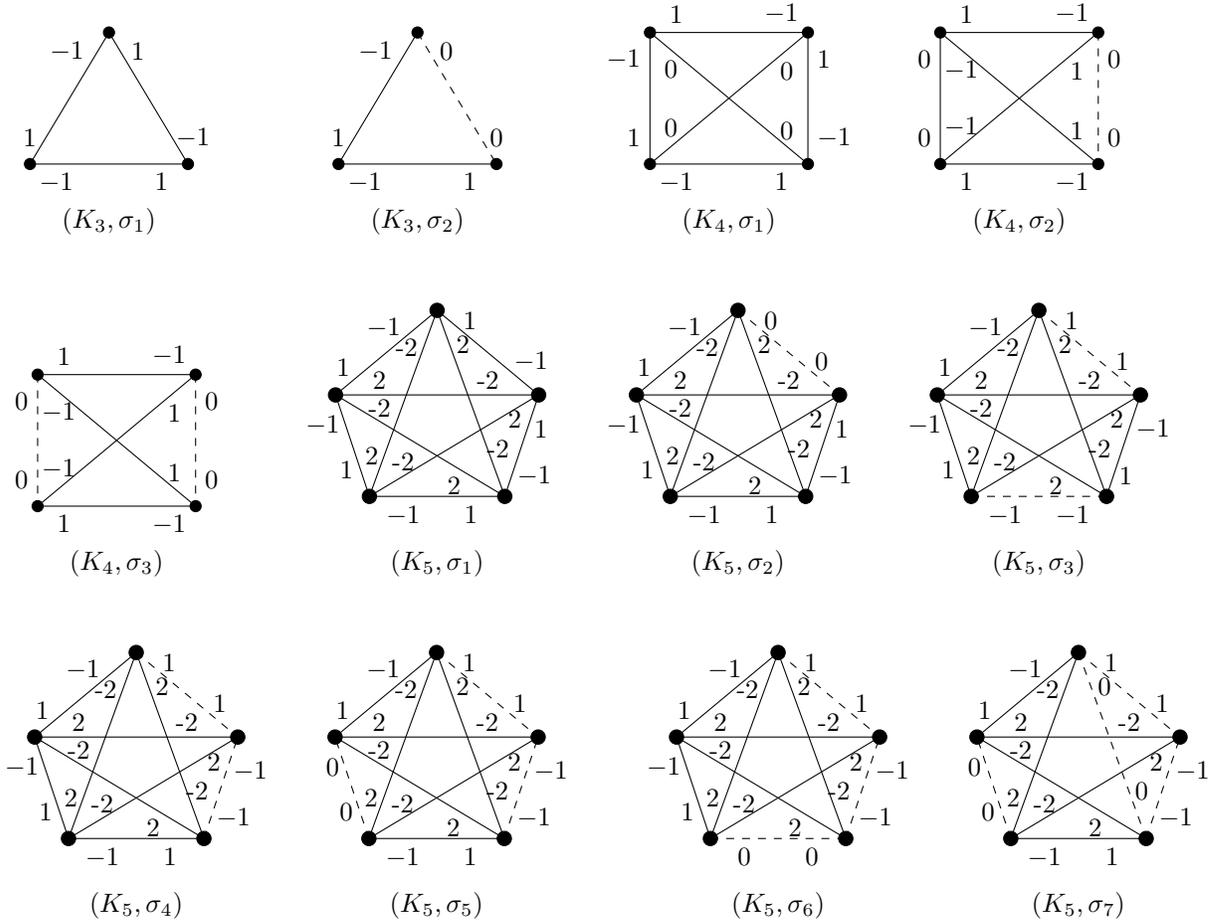

\begin{theorem}\label{Theorem: chro index of signed K_3s}
For any signature $\sigma$, $\chi ' (K_3, \sigma) = \begin{cases}
2,~\text{if}~\sigma \sim \sigma_1,\\
3,~\text{if}~\sigma \sim \sigma_2.
\end{cases}$
\end{theorem}
\begin{proof}
The proof directly follows from the Proposition~\ref{Prop-chro index of signed cycles}.
\end{proof}

However, signed graphs $(K_3, \sigma_1)$ and $(K_3, \sigma_2)$ and their proper edge colorings are shown in Figure~\ref{Figure-proper edge colorings of signed K_n up to 5}. 

\begin{theorem}\label{Theorem: chro index of signed K_4s}
For any signature $\sigma$, $\chi ' (K_4, \sigma) = 3$.
\end{theorem}
\begin{proof}
By Theorem~\ref{Vizing's thm}, $3 \leq \chi ' (K_4, \sigma) \leq 4$ for every signature $\sigma$. It is known that every signed $(K_4, \sigma)$ is switching equivalent to one of $(K_4, \sigma_1)$, $(K_4, \sigma_2)$ and $(K_4, \sigma_3)$ which are shown in Figure~\ref{Figure-proper edge colorings of signed K_n up to 5}. Also a proper 3-edge coloring of each $(K_4, \sigma_i)$, for $i=1,2,3$, is given in Figure~\ref{Figure-proper edge colorings of signed K_n up to 5}. This completes the proof.
\end{proof}

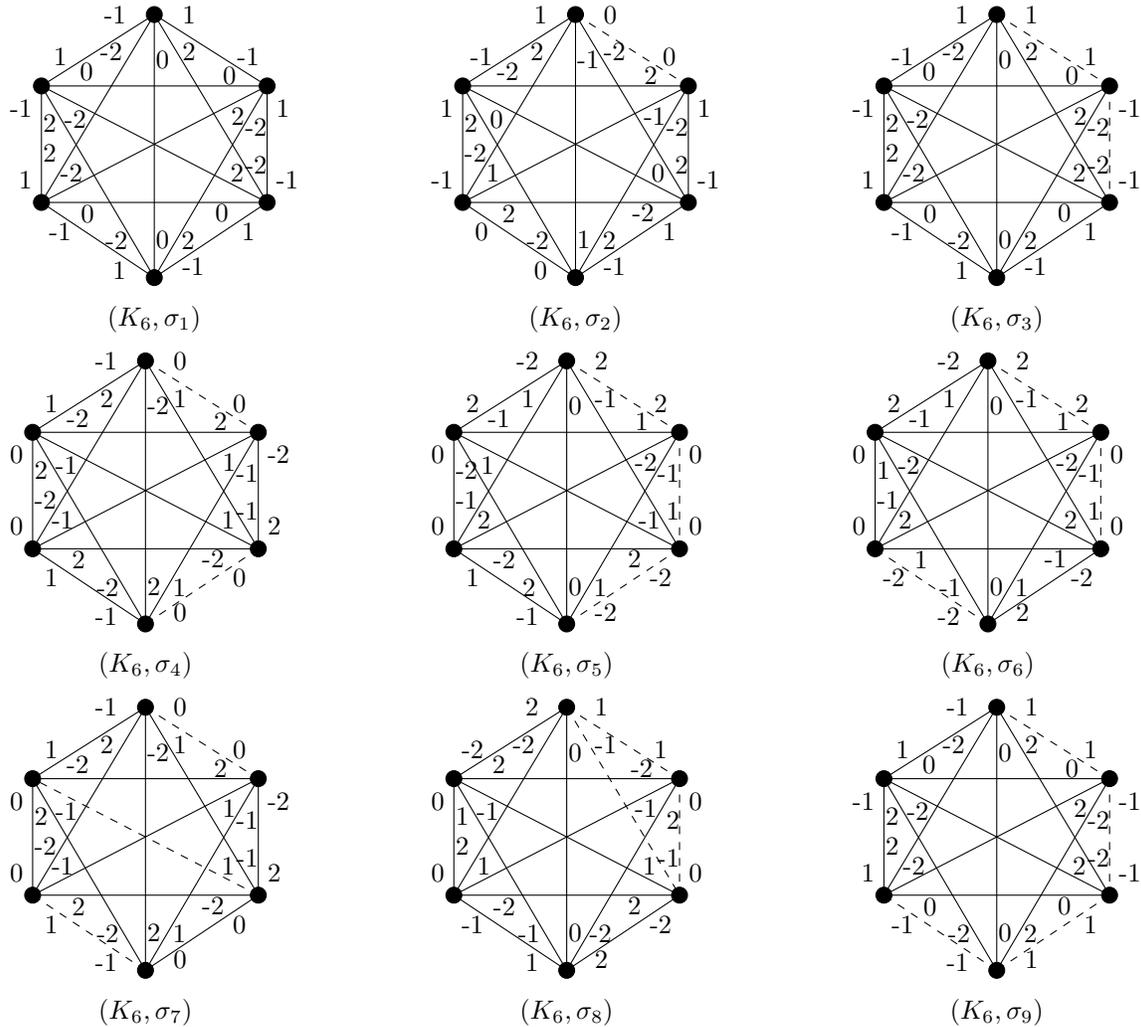
\begin{figure}[ht]
\begin{subfigure}{0.3\textwidth}
\begin{tikzpicture}[scale=0.5]

\draw[fill=black] (12,7) circle (6pt);
\draw[fill=black] (15,5.1) circle (6pt);
\draw[fill=black] (15,2) circle (6pt);
\draw[fill=black] (12,0) circle (6pt);
\draw[fill=black] (9,2) circle (6pt);
\node [below] at (12,-0.5) {$(K_{6},\sigma_1)$};
\draw[fill=black] (9,5.1) circle (6pt);

\node [right] at (12.5,7) {$1$};
\node [above] at (14.5,5.4) {-$1$};
\node [right] at (15,4.5) {$1$};
\node [right] at (15,2.6) {-$1$};
\node [below] at (14.5,1.7) {$1$};
\node [right] at (12.5,0.3) {-$1$};
\node [left] at (11.5,0.2) {$1$};
\node [below] at (9.5,1.7) {-$1$};
\node [left] at (9,2.6) {$1$};
\node [left] at (9,4.5) {-$1$};
\node [above] at (9.5,5.4) {$1$};
\node [left] at (11.5,7) {-$1$};

\node [right] at (12.5,6) {$2$};
\node [above] at (14,4.9) {$0$};
\node [right] at (11.8,5.8) {$0$};

\node [left] at (15.25,4) {-$2$};
\node [left] at (15.25,3) {-$2$};
\node [left] at (14.65,4.3) {$2$};

\node [below] at (13.8,2.2) {$0$};
\node [above] at (12.9,0.5) {$2$};
\node [above] at (14.2,2.3) {$2$};

\node [above] at (11,0.5) {-$2$};
\node [right] at (9.8,1.7) {$0$};
\node [right] at (11.8,1) {$0$};

\node [right] at (8.8,3.3) {$2$};
\node [right] at (8.8,4.1) {$2$};
\node [above] at (9.8,2.3) {-$2$};

\node [left] at (11.4,6) {-$2$};
\node [above] at (10.2,5) {$0$};
\node [below] at (9.9,4.7) {-$2$};

\draw (12,7) -- (15,5.1);
\draw (12,7) -- (15,2);
\draw (12,7) -- (12,0);
\draw (12,7) -- (9,2);
\draw (12,7) -- (9,5.1);
\draw (15,5.1) -- (15,2);
\draw (15,5.1) -- (12,0);
\draw (15,5.1) -- (9,2);
\draw (15,5.1) -- (9,5.1);
\draw (15,2) -- (12,0);
\draw (15,2) -- (9,2);
\draw (15,2) -- (9,5.1);
\draw (12,0) -- (9,2);
\draw (12,0) -- (9,5.1);
\draw (9,2) -- (9,5.1);

\end{tikzpicture}
\end{subfigure}
\hfill
\begin{subfigure}{0.3\textwidth}
\begin{tikzpicture}[scale=0.5]

\draw[fill=black] (12,7) circle (6pt);
\draw[fill=black] (15,5.1) circle (6pt);
\draw[fill=black] (15,2) circle (6pt);
\draw[fill=black] (12,0) circle (6pt);
\draw[fill=black] (9,2) circle (6pt);
\node [below] at (12,-0.5) {$(K_{6},\sigma_2)$};
\draw[fill=black] (9,5.1) circle (6pt);

\node [right] at (12.5,7) {$0$};
\node [above] at (14.5,5.4) {$0$};
\node [right] at (15,4.5) {$1$};
\node [right] at (15,2.6) {-$1$};
\node [below] at (14.5,1.7) {$1$};
\node [right] at (12.5,0.3) {-$1$};
\node [left] at (11.5,0.2) {$0$};
\node [below] at (9.5,1.7) {$0$};
\node [left] at (9,2.6) {-$1$};
\node [left] at (9,4.5) {$1$};
\node [above] at (9.5,5.4) {-$1$};
\node [left] at (11.5,7) {$1$};

\node [right] at (12.5,6) {-$2$};
\node [above] at (14,4.9) {$2$};
\node [right] at (11.8,5.8) {-$1$};

\node [left] at (15.25,4) {-$2$};
\node [left] at (15.25,3) {$2$};
\node [left] at (14.65,4.3) {-$1$};

\node [below] at (13.8,2.2) {-$2$};
\node [above] at (12.9,0.5) {$2$};
\node [above] at (14.2,2.3) {$0$};

\node [above] at (11,0.5) {-$2$};
\node [right] at (9.8,1.7) {$2$};
\node [right] at (11.8,1) {$1$};

\node [right] at (8.8,3.3) {-$2$};
\node [right] at (8.8,4.1) {$2$};
\node [above] at (9.8,2.3) {$1$};

\node [left] at (11.4,6) {$2$};
\node [above] at (10.2,5) {-$2$};
\node [below] at (9.9,4.7) {$0$};

\draw [dashed] (12,7) -- (15,5.1);
\draw (12,7) -- (15,2);
\draw (12,7) -- (12,0);
\draw (12,7) -- (9,2);
\draw (12,7) -- (9,5.1);
\draw (15,5.1) -- (15,2);
\draw (15,5.1) -- (12,0);
\draw (15,5.1) -- (9,2);
\draw (15,5.1) -- (9,5.1);
\draw (15,2) -- (12,0);
\draw (15,2) -- (9,2);
\draw (15,2) -- (9,5.1);
\draw (12,0) -- (9,2);
\draw (12,0) -- (9,5.1);
\draw (9,2) -- (9,5.1);

\end{tikzpicture}
\end{subfigure}
\hfill
\begin{subfigure}{0.3\textwidth}
\begin{tikzpicture}[scale=0.5]
\draw[fill=black] (12,7) circle (6pt);
\draw[fill=black] (15,5.1) circle (6pt);
\draw[fill=black] (15,2) circle (6pt);
\draw[fill=black] (12,0) circle (6pt);
\draw[fill=black] (9,2) circle (6pt);
\node [below] at (12,-0.5) {$(K_{6},\sigma_3)$};
\draw[fill=black] (9,5.1) circle (6pt);

\node [right] at (12.5,7) {$1$};
\node [above] at (14.5,5.4) {$1$};
\node [right] at (15,4.5) {-$1$};
\node [right] at (15,2.6) {-$1$};
\node [below] at (14.5,1.7) {$1$};
\node [right] at (12.5,0.3) {-$1$};
\node [left] at (11.5,0.2) {$1$};
\node [below] at (9.5,1.7) {-$1$};
\node [left] at (9,2.6) {$1$};
\node [left] at (9,4.5) {-$1$};
\node [above] at (9.5,5.4) {-$1$};
\node [left] at (11.5,7) {$1$};

\node [right] at (12.5,6) {$2$};
\node [above] at (14,4.9) {$0$};
\node [right] at (11.8,5.8) {$0$};

\node [left] at (15.25,4) {-$2$};
\node [left] at (15.25,3) {-$2$};
\node [left] at (14.65,4.3) {$2$};

\node [below] at (13.8,2.2) {$0$};
\node [above] at (12.9,0.5) {$2$};
\node [above] at (14.2,2.3) {$2$};

\node [above] at (11,0.5) {-$2$};
\node [right] at (9.8,1.7) {$0$};
\node [right] at (11.8,1) {$0$};

\node [right] at (8.8,3.3) {$2$};
\node [right] at (8.8,4.1) {$2$};
\node [above] at (9.8,2.3) {-$2$};

\node [left] at (11.4,6) {-$2$};
\node [above] at (10.2,5) {$0$};
\node [below] at (9.9,4.7) {-$2$};

\draw [dashed] (12,7) -- (15,5.1);
\draw (12,7) -- (15,2);
\draw (12,7) -- (12,0);
\draw (12,7) -- (9,2);
\draw (12,7) -- (9,5.1);
\draw [dashed] (15,5.1) -- (15,2);
\draw (15,5.1) -- (12,0);
\draw (15,5.1) -- (9,2);
\draw (15,5.1) -- (9,5.1);
\draw (15,2) -- (12,0);
\draw (15,2) -- (9,2);
\draw (15,2) -- (9,5.1);
\draw (12,0) -- (9,2);
\draw (12,0) -- (9,5.1);
\draw (9,2) -- (9,5.1);

\end{tikzpicture}
\end{subfigure}
\hfill
\begin{subfigure}{0.3\textwidth}
\begin{tikzpicture}[scale=0.5]
\draw[fill=black] (12,7) circle (6pt);
\draw[fill=black] (15,5.1) circle (6pt);
\draw[fill=black] (15,2) circle (6pt);
\draw[fill=black] (12,0) circle (6pt);
\draw[fill=black] (9,2) circle (6pt);
\node [below] at (12,-0.5) {$(K_{6},\sigma_4)$};
\draw[fill=black] (9,5.1) circle (6pt);

\node [right] at (12.5,7) {$0$};
\node [above] at (14.5,5.4) {$0$};
\node [right] at (15,4.5) {-$2$};
\node [right] at (15,2.6) {$2$};
\node [below] at (14.5,1.7) {$0$};
\node [right] at (12.5,0.3) {$0$};
\node [left] at (11.5,0.2) {-$1$};
\node [below] at (9.5,1.7) {$1$};
\node [left] at (9,2.6) {$0$};
\node [left] at (9,4.5) {$0$};
\node [above] at (9.5,5.4) {$1$};
\node [left] at (11.5,7) {-$1$};

\node [right] at (12.5,6) {$1$};
\node [above] at (14,4.9) {$2$};
\node [right] at (11.8,5.8) {-$2$};

\node [left] at (15.25,4) {-$1$};
\node [left] at (15.25,3) {-$1$};
\node [left] at (14.65,4.3) {$1$};

\node [below] at (13.8,2.2) {-$2$};
\node [above] at (12.9,0.5) {$1$};
\node [above] at (14.2,2.3) {$1$};

\node [above] at (11,0.5) {-$2$};
\node [right] at (9.8,1.7) {$2$};
\node [right] at (11.8,1) {$2$};

\node [right] at (8.8,3.3) {-$2$};
\node [right] at (8.8,4.1) {$2$};
\node [above] at (9.8,2.3) {-$1$};

\node [left] at (11.4,6) {$2$};
\node [above] at (10.2,5) {-$2$};
\node [below] at (9.9,4.7) {-$1$};

\draw [dashed] (12,7) -- (15,5.1);
\draw (12,7) -- (15,2);
\draw (12,7) -- (12,0);
\draw (12,7) -- (9,2);
\draw (12,7) -- (9,5.1);
\draw (15,5.1) -- (15,2);
\draw (15,5.1) -- (12,0);
\draw (15,5.1) -- (9,2);
\draw (15,5.1) -- (9,5.1);
\draw [dashed] (15,2) -- (12,0);
\draw (15,2) -- (9,2);
\draw (15,2) -- (9,5.1);
\draw (12,0) -- (9,2);
\draw (12,0) -- (9,5.1);
\draw (9,2) -- (9,5.1);

\end{tikzpicture}
\end{subfigure}
\hfill
\begin{subfigure}{0.3\textwidth}
\begin{tikzpicture}[scale=0.5]
\draw[fill=black] (12,7) circle (6pt);
\draw[fill=black] (15,5.1) circle (6pt);
\draw[fill=black] (15,2) circle (6pt);
\draw[fill=black] (12,0) circle (6pt);
\draw[fill=black] (9,2) circle (6pt);
\node [below] at (12,-0.5) {$(K_{6},\sigma_5)$};
\draw[fill=black] (9,5.1) circle (6pt);

\node [right] at (12.5,7) {$2$};
\node [above] at (14.5,5.4) {$2$};
\node [right] at (15,4.5) {$0$};
\node [right] at (15,2.6) {$0$};
\node [below] at (14.5,1.7) {-$2$};
\node [right] at (12.5,0.3) {-$2$};
\node [left] at (11.5,0.2) {-$1$};
\node [below] at (9.5,1.7) {$1$};
\node [left] at (9,2.6) {$0$};
\node [left] at (9,4.5) {$0$};
\node [above] at (9.5,5.4) {$2$};
\node [left] at (11.5,7) {-$2$};

\node [right] at (12.5,6) {-$1$};
\node [above] at (14,4.9) {$1$};
\node [right] at (11.8,5.8) {$0$};

\node [left] at (15.25,4) {-$1$};
\node [left] at (15.25,3) {$1$};
\node [left] at (14.65,4.3) {-$2$};

\node [below] at (13.8,2.2) {$2$};
\node [above] at (12.9,0.5) {$1$};
\node [above] at (14.2,2.3) {-$1$};

\node [above] at (11,0.5) {$2$};
\node [right] at (9.8,1.7) {-$2$};
\node [right] at (11.8,1) {$0$};

\node [right] at (8.8,3.3) {-$1$};
\node [right] at (8.8,4.1) {-$2$};
\node [above] at (9.8,2.3) {$2$};

\node [left] at (11.4,6) {$1$};
\node [above] at (10.2,5) {-$1$};
\node [below] at (9.9,4.7) {$1$};

\draw [dashed] (12,7) -- (15,5.1);
\draw (12,7) -- (15,2);
\draw (12,7) -- (12,0);
\draw (12,7) -- (9,2);
\draw (12,7) -- (9,5.1);
\draw [dashed] (15,5.1) -- (15,2);
\draw (15,5.1) -- (12,0);
\draw (15,5.1) -- (9,2);
\draw (15,5.1) -- (9,5.1);
\draw [dashed] (15,2) -- (12,0);
\draw (15,2) -- (9,2);
\draw (15,2) -- (9,5.1);
\draw (12,0) -- (9,2);
\draw (12,0) -- (9,5.1);
\draw (9,2) -- (9,5.1);

\end{tikzpicture}
\end{subfigure}
\hfill
\begin{subfigure}{0.3\textwidth}
\begin{tikzpicture}[scale=0.5]
\draw[fill=black] (12,7) circle (6pt);
\draw[fill=black] (15,5.1) circle (6pt);
\draw[fill=black] (15,2) circle (6pt);
\draw[fill=black] (12,0) circle (6pt);
\draw[fill=black] (9,2) circle (6pt);
\node [below] at (12,-0.5) {$(K_{6},\sigma_6)$};
\draw[fill=black] (9,5.1) circle (6pt);

\node [right] at (12.5,7) {$2$};
\node [above] at (14.5,5.4) {$2$};
\node [right] at (15,4.5) {$0$};
\node [right] at (15,2.6) {$0$};
\node [below] at (14.5,1.7) {-$2$};
\node [right] at (12.5,0.3) {$2$};
\node [left] at (11.5,0.2) {-$2$};
\node [below] at (9.5,1.7) {-$2$};
\node [left] at (9,2.6) {$0$};
\node [left] at (9,4.5) {$0$};
\node [above] at (9.5,5.4) {$2$};
\node [left] at (11.5,7) {-$2$};

\node [right] at (12.5,6) {-$1$};
\node [above] at (14,4.9) {$1$};
\node [right] at (11.8,5.8) {$0$};

\node [left] at (15.25,4) {-$1$};
\node [left] at (15.25,3) {$1$};
\node [left] at (14.65,4.3) {-$2$};

\node [below] at (13.8,2.2) {-$1$};
\node [above] at (12.9,0.5) {$1$};
\node [above] at (14.2,2.3) {$2$};

\node [above] at (11,0.5) {-$1$};
\node [right] at (9.8,1.7) {$1$};
\node [right] at (11.8,1) {$0$};

\node [right] at (8.8,3.3) {-$1$};
\node [right] at (8.8,4.1) {$1$};
\node [above] at (9.8,2.3) {$2$};

\node [left] at (11.4,6) {$1$};
\node [above] at (10.2,5) {-$1$};
\node [below] at (9.9,4.7) {-$2$};

\draw [dashed] (12,7) -- (15,5.1);
\draw (12,7) -- (15,2);
\draw (12,7) -- (12,0);
\draw (12,7) -- (9,2);
\draw (12,7) -- (9,5.1);
\draw [dashed] (15,5.1) -- (15,2);
\draw (15,5.1) -- (12,0);
\draw (15,5.1) -- (9,2);
\draw (15,5.1) -- (9,5.1);
\draw (15,2) -- (12,0);
\draw (15,2) -- (9,2);
\draw (15,2) -- (9,5.1);
\draw [dashed] (12,0) -- (9,2);
\draw (12,0) -- (9,5.1);
\draw (9,2) -- (9,5.1);
\end{tikzpicture}
\end{subfigure}
\hfill
\begin{subfigure}{0.3\textwidth}
\begin{tikzpicture}[scale=0.5]
\draw[fill=black] (12,7) circle (6pt);
\draw[fill=black] (15,5.1) circle (6pt);
\draw[fill=black] (15,2) circle (6pt);
\draw[fill=black] (12,0) circle (6pt);
\draw[fill=black] (9,2) circle (6pt);
\node [below] at (12,-0.5) {$(K_{6},\sigma_7)$};
\draw[fill=black] (9,5.1) circle (6pt);

\node [right] at (12.5,7) {$0$};
\node [above] at (14.5,5.4) {$0$};
\node [right] at (15,4.5) {-$2$};
\node [right] at (15,2.6) {$2$};
\node [below] at (14.5,1.7) {$0$};
\node [right] at (12.5,0.3) {$0$};
\node [left] at (11.5,0.2) {-$1$};
\node [below] at (9.5,1.7) {$1$};
\node [left] at (9,2.6) {$0$};
\node [left] at (9,4.5) {$0$};
\node [above] at (9.5,5.4) {$1$};
\node [left] at (11.5,7) {-$1$};

\node [right] at (12.5,6) {$1$};
\node [above] at (14,4.9) {$2$};
\node [right] at (11.8,5.8) {-$2$};

\node [left] at (15.25,4) {-$1$};
\node [left] at (15.25,3) {-$1$};
\node [left] at (14.65,4.3) {$1$};

\node [below] at (13.8,2.2) {-$2$};
\node [above] at (12.9,0.5) {$1$};
\node [above] at (14.2,2.3) {$1$};

\node [above] at (11,0.5) {-$2$};
\node [right] at (9.8,1.7) {$2$};
\node [right] at (11.8,1) {$2$};

\node [right] at (8.8,3.3) {-$2$};
\node [right] at (8.8,4.1) {$2$};
\node [above] at (9.8,2.3) {-$1$};

\node [left] at (11.4,6) {$2$};
\node [above] at (10.2,5) {-$2$};
\node [below] at (9.9,4.7) {-$1$};

\draw [dashed] (12,7) -- (15,5.1);
\draw (12,7) -- (15,2);
\draw (12,7) -- (12,0);
\draw (12,7) -- (9,2);
\draw (12,7) -- (9,5.1);
\draw (15,5.1) -- (15,2);
\draw (15,5.1) -- (12,0);
\draw (15,5.1) -- (9,2);
\draw (15,5.1) -- (9,5.1);
\draw (15,2) -- (12,0);
\draw (15,2) -- (9,2);
\draw [dashed] (15,2) -- (9,5.1);
\draw [dashed] (12,0) -- (9,2);
\draw (12,0) -- (9,5.1);
\draw (9,2) -- (9,5.1);

\end{tikzpicture}
\end{subfigure}
\hfill
\begin{subfigure}{0.3\textwidth}
\begin{tikzpicture}[scale=0.5]
\draw[fill=black] (12,7) circle (6pt);
\draw[fill=black] (15,5.1) circle (6pt);
\draw[fill=black] (15,2) circle (6pt);
\draw[fill=black] (12,0) circle (6pt);
\draw[fill=black] (9,2) circle (6pt);
\node [below] at (12,-0.5) {$(K_{6},\sigma_8)$};
\draw[fill=black] (9,5.1) circle (6pt);

\node [right] at (12.5,7) {$1$};
\node [above] at (14.5,5.4) {$1$};
\node [right] at (15,4.5) {$0$};
\node [right] at (15,2.6) {$0$};
\node [below] at (14.5,1.7) {-$2$};
\node [right] at (12.5,0.3) {$2$};
\node [left] at (11.5,0.2) {$1$};
\node [below] at (9.5,1.7) {-$1$};
\node [left] at (9,2.6) {$0$};
\node [left] at (9,4.5) {$0$};
\node [above] at (9.5,5.4) {-$2$};
\node [left] at (11.5,7) {$2$};

\node [right] at (12.5,6) {-$1$};
\node [above] at (14,4.9) {-$2$};
\node [right] at (11.8,5.8) {$0$};

\node [left] at (15.25,4) {$2$};
\node [left] at (15.25,3) {-$1$};
\node [left] at (14.65,4.3) {-$1$};

\node [below] at (13.8,2.2) {$2$};
\node [above] at (12.9,0.5) {-$2$};
\node [above] at (14.2,2.3) {$1$};

\node [above] at (11,0.5) {-$1$};
\node [right] at (9.8,1.7) {-$2$};
\node [right] at (11.8,1) {$0$};

\node [right] at (8.8,3.3) {$2$};
\node [right] at (8.8,4.1) {$1$};
\node [above] at (9.8,2.3) {$1$};

\node [left] at (11.4,6) {-$2$};
\node [above] at (10.2,5) {$2$};
\node [below] at (9.9,4.7) {-$1$};

\draw [dashed] (12,7) -- (15,5.1);
\draw [dashed] (12,7) -- (15,2);
\draw (12,7) -- (12,0);
\draw (12,7) -- (9,2);
\draw (12,7) -- (9,5.1);
\draw [dashed] (15,5.1) -- (15,2);
\draw (15,5.1) -- (12,0);
\draw (15,5.1) -- (9,2);
\draw (15,5.1) -- (9,5.1);
\draw (15,2) -- (12,0);
\draw (15,2) -- (9,2);
\draw (15,2) -- (9,5.1);
\draw (12,0) -- (9,2);
\draw (12,0) -- (9,5.1);
\draw (9,2) -- (9,5.1);

\end{tikzpicture}
\end{subfigure}
\hfill
\begin{subfigure}{0.3\textwidth}
\begin{tikzpicture}[scale=0.5]
\draw[fill=black] (12,7) circle (6pt);
\draw[fill=black] (15,5.1) circle (6pt);
\draw[fill=black] (15,2) circle (6pt);
\draw[fill=black] (12,0) circle (6pt);
\draw[fill=black] (9,2) circle (6pt);
\node [below] at (12,-0.5) {$(K_{6},\sigma_9)$};
\draw[fill=black] (9,5.1) circle (6pt);

\node [right] at (12.5,7) {$1$};
\node [above] at (14.5,5.4) {$1$};
\node [right] at (15,4.5) {-$1$};
\node [right] at (15,2.6) {-$1$};
\node [below] at (14.5,1.7) {$1$};
\node [right] at (12.5,0.3) {$1$};
\node [left] at (11.5,0.2) {-$1$};
\node [below] at (9.5,1.7) {-$1$};
\node [left] at (9,2.6) {$1$};
\node [left] at (9,4.5) {-$1$};
\node [above] at (9.5,5.4) {$1$};
\node [left] at (11.5,7) {-$1$};

\node [right] at (12.5,6) {$2$};
\node [above] at (14,4.9) {$0$};
\node [right] at (11.8,5.8) {$0$};

\node [left] at (15.25,4) {-$2$};
\node [left] at (15.25,3) {-$2$};
\node [left] at (14.65,4.3) {$2$};

\node [below] at (13.8,2.2) {$0$};
\node [above] at (12.9,0.5) {$2$};
\node [above] at (14.2,2.3) {$2$};

\node [above] at (11,0.5) {-$2$};
\node [right] at (9.8,1.7) {$0$};
\node [right] at (11.8,1) {$0$};

\node [right] at (8.8,3.3) {$2$};
\node [right] at (8.8,4.1) {$2$};
\node [above] at (9.8,2.3) {-$2$};

\node [left] at (11.4,6) {-$2$};
\node [above] at (10.2,5) {$0$};
\node [below] at (9.9,4.7) {-$2$};

\draw [dashed] (12,7) -- (15,5.1);
\draw (12,7) -- (15,2);
\draw (12,7) -- (12,0);
\draw (12,7) -- (9,2);
\draw (12,7) -- (9,5.1);
\draw [dashed] (15,5.1) -- (15,2);
\draw (15,5.1) -- (12,0);
\draw (15,5.1) -- (9,2);
\draw (15,5.1) -- (9,5.1);
\draw [dashed] (15,2) -- (12,0);
\draw (15,2) -- (9,2);
\draw (15,2) -- (9,5.1);
\draw [dashed] (12,0) -- (9,2);
\draw (12,0) -- (9,5.1);
\draw (9,2) -- (9,5.1);

\end{tikzpicture}
\end{subfigure}
\caption{Some switching non-isomorphic signed complete graphs over $K_6$ and their proper edge colorings.}
\label{Figure-1-SCGs on 6 vertices}
\end{figure}

\begin{prop}\label{Proposition: Prop required for coloring of signed K_5}
For any signature $\sigma$, if $\Sigma=(K_5,\sigma)$ and $\chi'(\Sigma) =4$, then $\Sigma_1$ and $\Sigma_2$ must be positive cycles of length five.
\end{prop}
\begin{proof}
Let $\sigma$ be a signature such that $\chi'(\Sigma) =4$ and suppose $\gamma$ is a corresponding proper 4-edge coloring of $\Sigma$, where $\Sigma=(K_5,\sigma)$. If the set of colors is $\{ \pm 1, \pm 2\}$ for coloring $\gamma$, then each component of $\Sigma_1[\gamma]$ and $\Sigma_2[\gamma]$ is either a path or a positive cycle (due to Observation 1). Thus the coloring $\gamma$ corresponds to a partition of the edges of $\Sigma$ into balanced subgraphs of maximum degree 2. Consequently, $\Sigma_1$ and $\Sigma_2$ must be positive cycles of length 5. 
\end{proof}

\begin{theorem}\label{Theorem: chro index of signed K_5s}
For any signature $\sigma$, $\chi ' (K_5, \sigma) = \begin{cases}
4,~\text{if}~\sigma \sim \sigma_i~\text{in Fig.~\ref{Figure-proper edge colorings of signed K_n up to 5} for}~i=1,3,4;\\
5,~\text{otherwise}.
\end{cases}$
\end{theorem}

\begin{figure}[ht]
\begin{subfigure}{0.3\textwidth}
\begin{tikzpicture}[scale=0.5]
\draw[fill=black] (12,7) circle (6pt);
\draw[fill=black] (15,5.1) circle (6pt);
\draw[fill=black] (15,2) circle (6pt);
\draw[fill=black] (12,0) circle (6pt);
\draw[fill=black] (9,2) circle (6pt);
\node [below] at (12,-0.5) {$(K_{6},\sigma_{10})$};
\draw[fill=black] (9,5.1) circle (6pt);

\node [right] at (12.5,7) {$1$};
\node [above] at (14.5,5.4) {$1$};
\node [right] at (15,4.5) {-$1$};
\node [right] at (15,2.6) {-$1$};
\node [below] at (14.5,1.7) {$1$};
\node [right] at (12.5,0.2) {$1$};
\node [left] at (11.5,0.2) {-$1$};
\node [below] at (9.5,1.7) {$1$};
\node [left] at (9,2.6) {-$1$};
\node [left] at (9,4.5) {-$1$};
\node [above] at (9.5,5.4) {$1$};
\node [left] at (11.5,7) {-$1$};

\node [right] at (12.5,6) {$2$};
\node [above] at (14,4.9) {$0$};
\node [right] at (11.9,5.8) {$0$};

\node [left] at (15.25,4) {-$2$};
\node [left] at (15.25,3) {-$2$};
\node [left] at (14.65,4.3) {$2$};

\node [below] at (13.8,2.2) {$0$};
\node [above] at (12.9,0.5) {$2$};
\node [above] at (14.2,2.3) {$2$};

\node [above] at (11,0.5) {-$2$};
\node [right] at (9.8,1.6) {$0$};
\node [right] at (11.8,1) {$0$};

\node [right] at (8.8,3.3) {$2$};
\node [right] at (8.8,4.1) {$2$};
\node [above] at (9.8,2.3) {-$2$};

\node [left] at (11.4,6) {-$2$};
\node [above] at (10.2,5) {$0$};
\node [below] at (9.9,4.7) {-$2$};

\draw [dashed] (12,7) -- (15,5.1);
\draw (12,7) -- (15,2);
\draw (12,7) -- (12,0);
\draw (12,7) -- (9,2);
\draw (12,7) -- (9,5.1);
\draw [dashed] (15,5.1) -- (15,2);
\draw (15,5.1) -- (12,0);
\draw (15,5.1) -- (9,2);
\draw (15,5.1) -- (9,5.1);
\draw [dashed] (15,2) -- (12,0);
\draw (15,2) -- (9,2);
\draw (15,2) -- (9,5.1);
\draw (12,0) -- (9,2);
\draw (12,0) -- (9,5.1);
\draw [dashed] (9,2) -- (9,5.1);

\end{tikzpicture}
\end{subfigure}
\hfill
\begin{subfigure}{0.3\textwidth}
\begin{tikzpicture}[scale=0.5]
\draw[fill=black] (12,7) circle (6pt);
\draw[fill=black] (15,5.1) circle (6pt);
\draw[fill=black] (15,2) circle (6pt);
\draw[fill=black] (12,0) circle (6pt);
\draw[fill=black] (9,2) circle (6pt);
\node [below] at (12,-0.5) {$(K_{6},\sigma_{11})$};
\draw[fill=black] (9,5.1) circle (6pt);

\node [right] at (12.5,7) {$1$};
\node [above] at (14.5,5.4) {$1$};
\node [right] at (15,4.5) {-$1$};
\node [right] at (15,2.6) {$1$};
\node [below] at (14.5,1.7) {-$1$};
\node [right] at (12.5,0.3) {-$1$};
\node [left] at (11.5,0.2) {$1$};
\node [below] at (9.5,1.7) {$1$};
\node [left] at (9,2.6) {-$1$};
\node [left] at (9,4.5) {$1$};
\node [above] at (9.5,5.4) {-$1$};
\node [left] at (11.5,7) {-$1$};

\node [right] at (12.5,6) {$2$};
\node [above] at (14,4.9) {$0$};
\node [right] at (11.8,5.8) {$0$};

\node [left] at (15.25,4) {-$2$};
\node [left] at (15.25,3) {-$2$};
\node [left] at (14.65,4.3) {$2$};

\node [below] at (13.8,2.2) {$0$};
\node [above] at (12.9,0.5) {$2$};
\node [above] at (14.2,2.3) {$2$};

\node [above] at (11,0.5) {-$2$};
\node [right] at (9.8,1.7) {$0$};
\node [right] at (11.8,1) {$0$};

\node [right] at (8.8,3.3) {$2$};
\node [right] at (8.8,4.1) {$2$};
\node [above] at (9.8,2.3) {-$2$};

\node [left] at (11.4,6) {-$2$};
\node [above] at (10.2,5) {$0$};
\node [below] at (9.9,4.7) {-$2$};

\draw [dashed] (12,7) -- (15,5.1);
\draw (12,7) -- (15,2);
\draw (12,7) -- (12,0);
\draw (12,7) -- (9,2);
\draw [dashed] (12,7) -- (9,5.1);
\draw (15,5.1) -- (15,2);
\draw (15,5.1) -- (12,0);
\draw (15,5.1) -- (9,2);
\draw (15,5.1) -- (9,5.1);
\draw [dashed] (15,2) -- (12,0);
\draw (15,2) -- (9,2);
\draw (15,2) -- (9,5.1);
\draw [dashed] (12,0) -- (9,2);
\draw (12,0) -- (9,5.1);
\draw (9,2) -- (9,5.1);

\end{tikzpicture}
\end{subfigure}
\hfill
\begin{subfigure}{0.3\textwidth}
\begin{tikzpicture}[scale=0.5]
\draw[fill=black] (12,7) circle (6pt);
\draw[fill=black] (15,5.1) circle (6pt);
\draw[fill=black] (15,2) circle (6pt);
\draw[fill=black] (12,0) circle (6pt);
\draw[fill=black] (9,2) circle (6pt);
\node [below] at (12,-0.5) {$(K_{6},\sigma_{12})$};
\draw[fill=black] (9,5.1) circle (6pt);

\node [right] at (12.5,7) {$1$};
\node [above] at (14.5,5.4) {$1$};
\node [right] at (15,4.5) {$0$};
\node [right] at (15,2.6) {$0$};
\node [below] at (14.5,1.7) {-$2$};
\node [right] at (12.5,0.3) {$2$};
\node [left] at (11.5,0.2) {$1$};
\node [below] at (9.5,1.7) {-$1$};
\node [left] at (9,2.6) {$0$};
\node [left] at (9,4.5) {$0$};
\node [above] at (9.5,5.4) {-$2$};
\node [left] at (11.5,7) {$2$};

\node [right] at (12.5,6) {-$1$};
\node [above] at (14,4.9) {-$2$};
\node [right] at (11.8,5.8) {$0$};

\node [left] at (15.25,4) {$2$};
\node [left] at (15.25,3) {-$1$};
\node [left] at (14.65,4.3) {-$1$};

\node [below] at (13.8,2.2) {$2$};
\node [above] at (12.9,0.5) {-$2$};
\node [above] at (14.2,2.3) {$1$};

\node [above] at (11,0.5) {-$1$};
\node [right] at (9.8,1.7) {-$2$};
\node [right] at (11.8,1) {$0$};

\node [right] at (8.8,3.3) {$2$};
\node [right] at (8.8,4.1) {$1$};
\node [above] at (9.8,2.3) {$1$};

\node [left] at (11.4,6) {-$2$};
\node [above] at (10.2,5) {$2$};
\node [below] at (9.9,4.7) {-$1$};

\draw [dashed] (12,7) -- (15,5.1);
\draw [dashed] (12,7) -- (15,2);
\draw (12,7) -- (12,0);
\draw (12,7) -- (9,2);
\draw (12,7) -- (9,5.1);
\draw [dashed] (15,5.1) -- (15,2);
\draw (15,5.1) -- (12,0);
\draw (15,5.1) -- (9,2);
\draw (15,5.1) -- (9,5.1);
\draw (15,2) -- (12,0);
\draw (15,2) -- (9,2);
\draw (15,2) -- (9,5.1);
\draw (12,0) -- (9,2);
\draw (12,0) -- (9,5.1);
\draw [dashed] (9,2) -- (9,5.1);

\end{tikzpicture}
\end{subfigure}
\hfill
\begin{subfigure}{0.3\textwidth}
\begin{tikzpicture}[scale=0.5]
\draw[fill=black] (12,7) circle (6pt);
\draw[fill=black] (15,5.1) circle (6pt);
\draw[fill=black] (15,2) circle (6pt);
\draw[fill=black] (12,0) circle (6pt);
\draw[fill=black] (9,2) circle (6pt);
\node [below] at (12,-0.5) {$(K_{6},\sigma_{13})$};
\draw[fill=black] (9,5.1) circle (6pt);

\node [right] at (12.5,7) {-$1$};
\node [above] at (14.5,5.4) {-$1$};
\node [right] at (15,4.5) {$1$};
\node [right] at (15,2.6) {$1$};
\node [below] at (14.5,1.7) {$2$};
\node [right] at (12.5,0.3) {-$2$};
\node [left] at (11.5,0.2) {$0$};
\node [below] at (9.5,1.7) {$0$};
\node [left] at (9,2.6) {-$1$};
\node [left] at (9,4.5) {$1$};
\node [above] at (9.5,5.4) {$2$};
\node [left] at (11.5,7) {$2$};

\node [right] at (12.5,6) {$0$};
\node [above] at (14,4.9) {$0$};
\node [right] at (11.8,5.8) {$1$};

\node [left] at (15.25,4) {-$2$};
\node [left] at (15.25,3) {$0$};
\node [left] at (14.65,4.3) {$2$};

\node [below] at (13.8,2.2) {-$1$};
\node [above] at (12.9,0.5) {$2$};
\node [above] at (14.2,2.3) {-$2$};

\node [above] at (11,0.5) {$1$};
\node [right] at (9.8,1.7) {$1$};
\node [right] at (11.8,1) {-$1$};

\node [right] at (8.8,3.3) {$2$};
\node [right] at (8.8,4.1) {-$1$};
\node [above] at (9.8,2.3) {-$2$};

\node [left] at (11.4,6) {-$2$};
\node [above] at (10.2,5) {$0$};
\node [below] at (9.9,4.7) {-$2$};

\draw [dashed] (12,7) -- (15,5.1);
\draw (12,7) -- (15,2);
\draw (12,7) -- (12,0);
\draw (12,7) -- (9,2);
\draw [dashed] (12,7) -- (9,5.1);
\draw [dashed] (15,5.1) -- (15,2);
\draw (15,5.1) -- (12,0);
\draw (15,5.1) -- (9,2);
\draw (15,5.1) -- (9,5.1);
\draw (15,2) -- (12,0);
\draw (15,2) -- (9,2);
\draw [dashed] (15,2) -- (9,5.1);
\draw (12,0) -- (9,2);
\draw (12,0) -- (9,5.1);
\draw (9,2) -- (9,5.1);

\end{tikzpicture}
\end{subfigure}
\hfill
\begin{subfigure}{0.3\textwidth}
\begin{tikzpicture}[scale=0.5]
\draw[fill=black] (12,7) circle (6pt);
\draw[fill=black] (15,5.1) circle (6pt);
\draw[fill=black] (15,2) circle (6pt);
\draw[fill=black] (12,0) circle (6pt);
\draw[fill=black] (9,2) circle (6pt);
\node [below] at (12,-0.5) {$(K_{6},\sigma_{14})$};
\draw[fill=black] (9,5.1) circle (6pt);

\node [right] at (12.5,7) {$1$};
\node [above] at (14.5,5.4) {$1$};
\node [right] at (15,4.5) {-$1$};
\node [right] at (15,2.6) {-$1$};
\node [below] at (14.5,1.7) {$1$};
\node [right] at (12.5,0.3) {-$1$};
\node [left] at (11.5,0.2) {$0$};
\node [below] at (9.5,1.7) {$0$};
\node [left] at (9,2.6) {-$1$};
\node [left] at (9,4.5) {-$1$};
\node [above] at (9.5,5.4) {-$2$};
\node [left] at (11.5,7) {$2$};

\node [right] at (12.5,6) {$0$};
\node [above] at (14,4.9) {$0$};
\node [right] at (11.8,5.8) {-$2$};

\node [left] at (15.25,4) {$2$};
\node [left] at (15.25,3) {$0$};
\node [left] at (14.65,4.3) {-$2$};

\node [below] at (13.8,2.2) {$2$};
\node [above] at (12.9,0.5) {-$2$};
\node [above] at (14.2,2.3) {-$2$};

\node [above] at (11,0.5) {$1$};
\node [right] at (9.8,1.7) {-$2$};
\node [right] at (11.8,1) {$2$};

\node [right] at (8.8,3.3) {$1$};
\node [right] at (8.8,4.1) {$1$};
\node [above] at (9.8,2.3) {$2$};

\node [left] at (11.4,6) {-$1$};
\node [above] at (10.2,5) {$0$};
\node [below] at (9.9,4.7) {$2$};

\draw [dashed] (12,7) -- (15,5.1);
\draw [dashed] (12,7) -- (15,2);
\draw (12,7) -- (12,0);
\draw (12,7) -- (9,2);
\draw (12,7) -- (9,5.1);
\draw [dashed] (15,5.1) -- (15,2);
\draw (15,5.1) -- (12,0);
\draw (15,5.1) -- (9,2);
\draw (15,5.1) -- (9,5.1);
\draw (15,2) -- (12,0);
\draw (15,2) -- (9,2);
\draw (15,2) -- (9,5.1);
\draw (12,0) -- (9,2);
\draw [dashed] (12,0) -- (9,5.1);
\draw [dashed] (9,2) -- (9,5.1);

\end{tikzpicture}
\end{subfigure}
\hfill
\begin{subfigure}{0.3\textwidth}
\begin{tikzpicture}[scale=0.5]
\draw[fill=black] (12,7) circle (6pt);
\draw[fill=black] (15,5.1) circle (6pt);
\draw[fill=black] (15,2) circle (6pt);
\draw[fill=black] (12,0) circle (6pt);
\draw[fill=black] (9,2) circle (6pt);
\node [below] at (12,-0.5) {$(K_{6},\sigma_{15})$};
\draw[fill=black] (9,5.1) circle (6pt);

\node [right] at (12.5,7) {$1$};
\node [above] at (14.5,5.4) {$1$};
\node [right] at (15,4.5) {-$1$};
\node [right] at (15,2.6) {-$1$};
\node [below] at (14.5,1.7) {$1$};
\node [right] at (12.5,0.3) {$1$};
\node [left] at (11.5,0.2) {-$1$};
\node [below] at (9.5,1.7) {-$1$};
\node [left] at (9,2.6) {$1$};
\node [left] at (9,4.5) {-$1$};
\node [above] at (9.5,5.4) {$1$};
\node [left] at (11.5,7) {-$1$};

\node [right] at (12.5,6) {$2$};
\node [above] at (14,4.9) {$2$};
\node [right] at (11.8,5.8) {-$2$};

\node [left] at (15.25,4) {$0$};
\node [left] at (15.25,3) {-$2$};
\node [left] at (14.65,4.3) {-$2$};

\node [below] at (13.8,2.2) {$2$};
\node [above] at (12.9,0.5) {$0$};
\node [above] at (14.2,2.3) {$0$};

\node [above] at (11,0.5) {-$2$};
\node [right] at (9.8,1.7) {-$2$};
\node [right] at (11.8,1) {$2$};

\node [right] at (8.8,3.3) {$0$};
\node [right] at (8.8,4.1) {$2$};
\node [above] at (9.8,2.3) {$2$};

\node [left] at (11.4,6) {$0$};
\node [above] at (10.2,5) {-$2$};
\node [below] at (9.9,4.7) {$0$};

\draw [dashed] (12,7) -- (15,5.1);
\draw (12,7) -- (15,2);
\draw (12,7) -- (12,0);
\draw [dashed] (12,7) -- (9,2);
\draw (12,7) -- (9,5.1);
\draw [dashed] (15,5.1) -- (15,2);
\draw (15,5.1) -- (12,0);
\draw (15,5.1) -- (9,2);
\draw (15,5.1) -- (9,5.1);
\draw [dashed] (15,2) -- (12,0);
\draw (15,2) -- (9,2);
\draw (15,2) -- (9,5.1);
\draw [dashed] (12,0) -- (9,2);
\draw (12,0) -- (9,5.1);
\draw (9,2) -- (9,5.1);

\end{tikzpicture}
\end{subfigure}
\hfill
\begin{subfigure}{0.3\textwidth}
\begin{tikzpicture}[scale=0.5]
\draw[fill=black] (12,7) circle (6pt);
\draw[fill=black] (15,5.1) circle (6pt);
\draw[fill=black] (15,2) circle (6pt);
\draw[fill=black] (12,0) circle (6pt);
\draw[fill=black] (9,2) circle (6pt);
\node [below] at (12,-0.5) {$(K_{6},\sigma_{16})$};
\draw[fill=black] (9,5.1) circle (6pt);

\node [right] at (12.5,7) {$2$};
\node [above] at (14.5,5.4) {$2$};
\node [right] at (15,4.5) {-$2$};
\node [right] at (15,2.6) {-$2$};
\node [below] at (14.5,1.7) {$2$};
\node [right] at (12.5,0.3) {-$2$};
\node [left] at (11.5,0.2) {$2$};
\node [below] at (9.5,1.7) {$2$};
\node [left] at (9,2.6) {-$2$};
\node [left] at (9,4.5) {-$2$};
\node [above] at (9.5,5.4) {$2$};
\node [left] at (11.5,7) {-$2$};

\node [right] at (12.5,6) {$0$};
\node [above] at (14,4.9) {-$1$};
\node [right] at (11.8,5.8) {-$1$};

\node [left] at (15.25,4) {$1$};
\node [left] at (15.25,3) {$0$};
\node [left] at (14.65,4.3) {$0$};

\node [below] at (13.8,2.2) {-$1$};
\node [above] at (12.9,0.5) {-$1$};
\node [above] at (14.2,2.3) {$1$};

\node [above] at (11,0.5) {$0$};
\node [right] at (9.8,1.7) {$1$};
\node [right] at (11.8,1) {$1$};

\node [right] at (8.8,3.3) {-$1$};
\node [right] at (8.8,4.1) {$0$};
\node [above] at (9.8,2.3) {$0$};

\node [left] at (11.4,6) {$1$};
\node [above] at (10.2,5) {$1$};
\node [below] at (9.9,4.7) {-$1$};

\draw [dashed] (12,7) -- (15,5.1);
\draw [dashed] (12,7) -- (15,2);
\draw (12,7) -- (12,0);
\draw (12,7) -- (9,2);
\draw (12,7) -- (9,5.1);
\draw [dashed] (15,5.1) -- (15,2);
\draw (15,5.1) -- (12,0);
\draw (15,5.1) -- (9,2);
\draw (15,5.1) -- (9,5.1);
\draw (15,2) -- (12,0);
\draw (15,2) -- (9,2);
\draw (15,2) -- (9,5.1);
\draw [dashed] (12,0) -- (9,2);
\draw [dashed] (12,0) -- (9,5.1);
\draw [dashed] (9,2) -- (9,5.1);

\end{tikzpicture}
\end{subfigure}
\caption{A continuation of Figure~\ref{Figure-1-SCGs on 6 vertices}. Remaining switching non-isomorphic signed complete graphs over $K_6$ (not shown in Figure~\ref{Figure-1-SCGs on 6 vertices}) and their proper edge colorings.}
\label{Figure-2-SCGs on 6 vertices}
\end{figure}
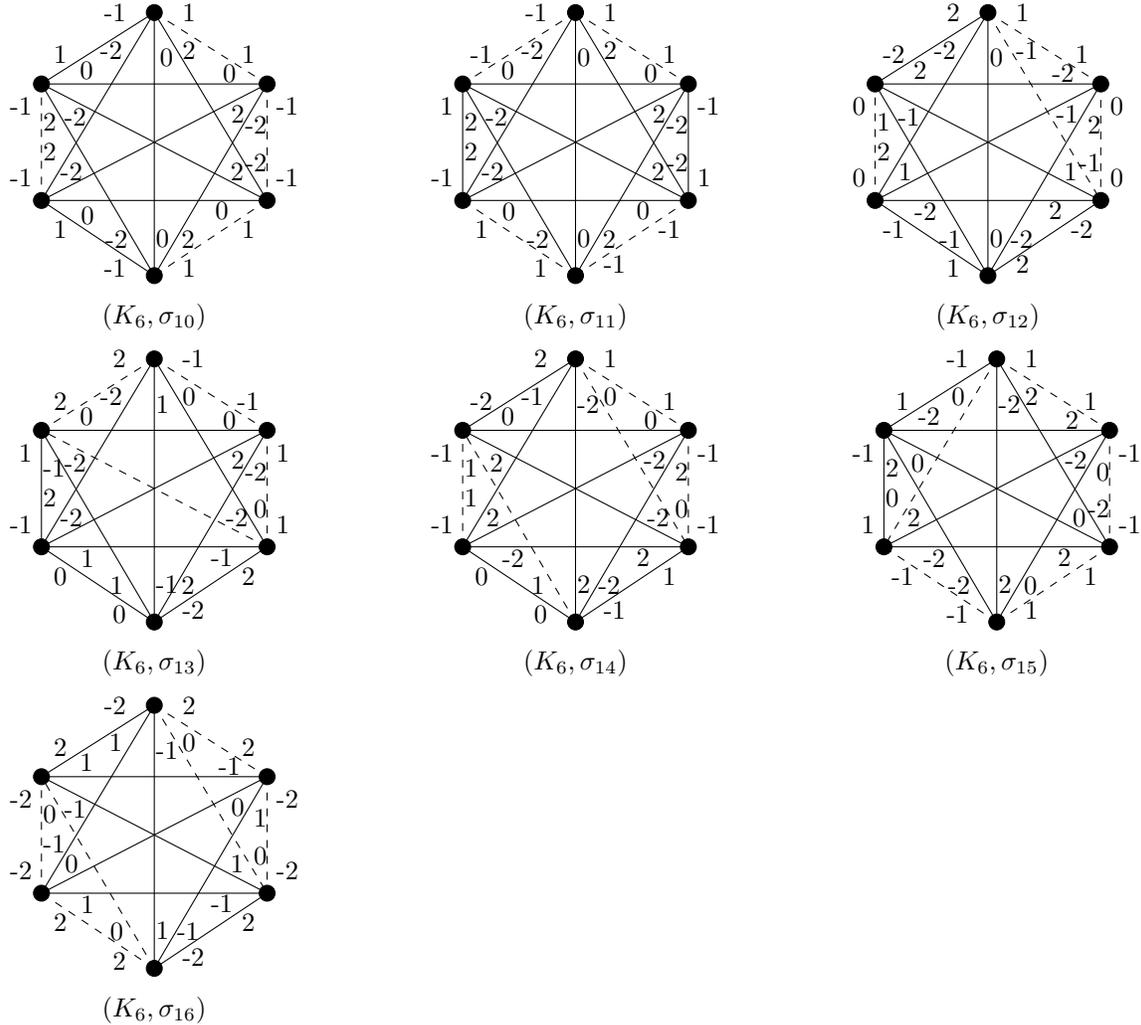

\begin{proof}
By Theorem~\ref{Vizing's thm}, $4 \leq \chi ' (K_5, \sigma) \leq 5$ for every signature $\sigma$. Therefore if $\sigma \sim \sigma_i$ for $i=1,3,4$, then to prove $\chi ' (K_5, \sigma) = 4$ , it suffices to give a proper 4-coloring of $(K_5, \sigma)$. Such a proper 4-coloring of $(K_5, \sigma_i)$ is given in Figure~\ref{Figure-proper edge colorings of signed K_n up to 5} for each $i=1,3,4$. Hence $\chi ' (K_5, \sigma) =4$ for $\sigma \sim \sigma_i$, where $i=1,3,4$.

For $\sigma \sim \sigma_2$, if $\chi ' (K_5, \sigma_2) =4$, then by Proposition~\ref{Proposition: Prop required for coloring of signed K_5}, $\Sigma_1$ and $\Sigma_2$ must be positive cycles of length 5. But $(K_5, \sigma_2)$ has exactly one negative edge, namely, $u_1u_2$ and therefore either $\Sigma_1$ or $\Sigma_2$ has to consist $u_1u_2$. Consequently, either $\Sigma_1$ or $\Sigma_2$ is a negative cycle, a contradiction. This proves that $\chi ' (K_5, \sigma_2) = 5$.

Similarly, for $j=5,6$, if $\sigma \sim \sigma_j$ and $\chi ' (K_5, \sigma) =4$, then by Proposition~\ref{Proposition: Prop required for coloring of signed K_5}, $\Sigma_1$ and $\Sigma_2$ must be positive cycles of length 5. But this is not possible because signed graphs $(K_5, \sigma_5)$ and $(K_5, \sigma_6)$ have exactly 3 negative edges making either $\Sigma_1$ or $\Sigma_2$ negative. Consequently $\chi ' (K_5, \sigma_j) = 5$ for $j=5,6$.

In $(K_5, \sigma_7)$, it is easy to verify that every $C_5$ is negative. Thus due to Proposition~\ref{Proposition: Prop required for coloring of signed K_5}, the chromatic index of $(K_5, \sigma_7)$ cannot be 4. Hence $\chi ' (K_5, \sigma_7) = 5$. This completes the proof.
\end{proof}

\begin{theorem}\label{Theorem: chro index of signed K_6s}
For any signature $\sigma$, $\chi ' (K_6, \sigma) = 5$.
\end{theorem}
\begin{proof}
By Theorem~\ref{Vizing's thm}, $5 \leq \chi ' (K_6, \sigma) \leq 6$ for every signature $\sigma$. It is clear that every signed $(K_6, \sigma)$ is switching equivalent to one of the signed complete graphs given in Figures~\ref{Figure-1-SCGs on 6 vertices} and \ref{Figure-2-SCGs on 6 vertices}. Thus, to complete the proof, it is sufficient to give a proper 5-edge coloring of each $(K_6, \sigma_i)$ for $i=1,\ldots ,16.$ A proper 5-edge coloring of each $(K_6, \sigma_i)$ is given in Figures~\ref{Figure-1-SCGs on 6 vertices} and \ref{Figure-2-SCGs on 6 vertices}. This completes the proof.
\end{proof}

In this subsection, we studied edge coloring of signed complete graphs of order up to 6. Further, we observed that the chromatic index of a signed complete graph $(K_n,\sigma)$ depends on the value of $n$ as well as signature $\sigma$. However, in case of even values of $n$ (here $n=4,6$), the chromatic index does not depend on the signature. So, based on this computation, we make the following conjecture.

\begin{conj}
Let $n \geq 8$ be an even number and let $\sigma$ be any signature, then $\chi ' (K_n, \sigma) = n-1$.
\end{conj}


\begin{thebibliography}{9}
\bibitem{Behr2020}
R. Behr, \textit{Edge coloring signed graphs}, Discrete Math. \textbf{343} (2020), 111654.

\bibitem{Cai2022}
H. Cai, Q. Sun, G. Xu and S. Zheng, \textit{Edge Coloring of the Signed Generalized Petersen Graph}, Bull. Malays. Math. Sci. Soc. \textbf{45} (2022), 647–661.

\bibitem{Harary1955}
F. Harary, \textit{On the notion of balance of a signed graph}, Mich. Math. J. \textbf{2} (1955), 143-146.

\bibitem{Janczewski2023}
R. Janczewski, K. Turowski and B. Wróblewski, \textit{Edge coloring of graphs of signed class 1 and 2}, Discret Appl. Math. \textbf{338} (2023), 311–319.

\bibitem{Máčajová2016}
E. Máčajová, A. Raspaud and M. Škoviera, \textit{The chromatic number of a signed graph}, Electron. J. Comb. \textbf{23} (2016), no. 1, 1–10.

\bibitem{Sehrawat2022}
D. Sehrawat and B. Bhattacharjya,  \textit{Chromatic Polynomials of Signed Book Graphs}, Theory Appl. Graphs, \textbf{9} (2022), no. 1, $\#4$.

\bibitem{Sehrawat2025}
D. Sehrawat, \textit{On double domination numbers of signed complete graphs}, Discret. Math. Algor. Appl. \textbf{17} (2025), no. 6, 2450094. 

\bibitem{Shi2007}
L. Shi and Z. Song, \textit{Upper bounds on the spectral radius of book-free and/or $K_{2,1}$-free graphs}, Linear Algebra Appl. \textbf{420} (2007), 526–529.


\bibitem{Vizing1964}
V. Vizing, \textit{On an estimate of the chromatic class of a p-graph}, Diskretn. Anal. \textbf{3} (1964), 23–30.

\bibitem{Wen2025}
C. Wen, Q. Sun, H. Cai and C. Zhang, \textit{The edge coloring of the Cartesian product of signed graphs}, Discrete Math. \textbf{348} (2025), 114276.

\bibitem{Zaslavsky1982-1}
T. Zaslavsky, \textit{Signed graphs}, Discret. Appl. Math. \textbf{4} (1982), 47–74.

\bibitem{Zaslavsky1982}
T. Zaslavsky, \textit{Signed graph coloring}, Discrete Math. \textbf{39} (1982), no. 2, 215–228.

\bibitem{Zhang2020}
L. Zhang, Y. Lu, R. Luo, D. Ye and S. Zhang, \textit{Edge coloring of signed graphs}, Discret. Appl. Math. \textbf{282} (2020), 234–242.

\bibitem{Zheng2022}
S. Zheng, H. Cai, Y. Wang and  Q. Sun, \textit{On the Chromatic Index of the Signed Generalized Petersen Graph $GP(n, 2)$} Axioms \textbf{11} (2022), no. 8, 393. 

\end{thebibliography}
\end{document}